\documentclass[11pt]{smfart}
\usepackage[T1]{fontenc}
\usepackage[utf8]{inputenc}

\usepackage{smfthm}%amsmath,amssymb}
\usepackage[french]{babel}

\usepackage{graphicx}
\usepackage{refcount}
\newcommand{\nextpageref}[1]{%
  \number\numexpr\getpagerefnumber{#1}+1\relax}
  
\input{macros-perso-fm.sty}
\input{macros-perso-fm-thm.sty}

\author{Frédéric Mangolte}
\address{LUNAM Universit\'e, LAREMA, Universit\'e d'Angers} \email{frederic.mangolte@univ-angers.fr}
\urladdr{http://www.math.univ-angers.fr/~mangolte}
\title[Variétés réelles de dimension $3$]{Topologie des variétés algébriques réelles de dimension~$3$} \alttitle{Real Algebraic Varieties}
\date{\today}
\makeindex

\begin{document}

\frontmatter
%\begin{abstract}
%
%\end{abstract}
%\begin{altabstract}
% \end{altabstract}
%\subjclass{14P}
%\keywords{Variété algébrique réelle}
% \altkeywords{Real Algebraic Variety}
%%\translator{⟨Pr ́enom Nom⟩} \thanks{⟨Subventions⟩} 
%\dedicatory{}
\maketitle
\setcounter{tocdepth}{2}     % Dans la table des matieres

\tableofcontents 

%%%%%%%%%%%%%%%%%%%%%%%%%%%%%%%%%%%%%%%%%%%%%
\mainmatter
%Corps de l’ouvrage

\section{Introduction : Modèles algébriques des variétés lisses}

%\nocite{Sc74}%[VI~2.2]
%\nocite{Sc94}
%\nocite{Ko01real} 
%\marginpar{$\approx$ pour difféo et $\cong$ pour isom alg, $\sim$ pour équivalence}

On sait grâce à Nash
\cite{Nash52} et Tognoli \cite{To73}, que toute variété~$\cC^\infty$ compacte connexe et sans bord admet un modèle algébrique réel. Plus précisément, soit $M$ une telle variété, il existe des polynômes réels 
$
P_1(x_1,\dots ,x_m), \dots,P_r(x_1,\dots ,x_m)
$ 
tels que le lieu de leurs zéros communs 
$$
X(\RR):= \{x\in\RR^m \textrm{ tels que }
P_1(x)=\dots = P_r(x)=0\}
$$
est lisse\footnote{Grâce au théorème de résolution des singularités d'Hironaka, on peut même supposer que le lieu des zéros complexes (réels et non réels) est lisse.} et difféomorphe à $M$. À la suite de son théorème, Nash conclut son article \cite{Nash52} par une conjecture selon laquelle la variété définie par les polynômes $
P_1, \dots,P_r
$ peut être choisie rationnelle (Définition~\ref{dfn.ratio}).

Dorénavant il est prouvé que cette conjecture est fausse en toute dimension plus grande que un.
En dimension~$2$, le Théorème de Comessatti~\cite{Co14}  affirme qu'aucun modèle algébrique réel d'une surface hyperbolique \emph{orientable} ne peut être une surface projective rationnelle non singulière (Théorème~\ref{thm.co}). En dimension~3, 
Koll\'ar\footnote{J\'anos Koll\'ar a reçu de l'AMS le prix Cole d'algèbre en 2006 pour ses travaux sur la conjecture de Nash et sur les variétés rationnellement connexes.} a montré en 1998 que seul un nombre fini de variétés hyperboliques pouvaient être rationnelles (Théorème~\ref{thm.kollar}). 
Ce résultat a été amélioré par  le Théorème de Eliashberg-Viterbo (2000) (Théorème~\ref{cor.vit}) qui impose qu'une variété hyperbolique de dimension $3$ ou supérieure ne peut être composante connexe d'une variété projective rationnelle non singulière. 
D'autres versions de la conjecture initiale ont été explorées en affaiblissant les hypothèses sur la variété algébrique. %, nous reviendrons sur ces versions dans la section~\ref{sec.conjnash}.
La question initiale nous amène donc naturellement à la recherche des types topologiques possibles du lieu réel (cf. page~\pageref{p.lieu}) d'une variété projective rationnelle non singulière. Plus généralement, on cherche à classifier les types topologiques des variétés rationnellement connexes (\ref{dfn.rc}) et uniréglées (\ref{dfn.uni}), ces notions s'ordonnant de la manière suivante: 
$$
\textrm{rationnel} \Rightarrow \textrm{rationnellement connexe} \Rightarrow \textrm{uniréglé}.
$$

Dans la section~\ref{sec.conjnash}, nous essayerons d'expliquer et de motiver les énoncés précédents et de donner une idée de la preuve de certains d'entre eux. Après avoir ainsi discuté des différentes variantes de la conjecture de Nash, nous rappelerons en section~\ref{sec.surf} la classification bien connue des surfaces rationnelles réelles. Puis nous présenterons en section~\ref{sec.uni}, l'état actuel de la classification des variétés projectives réelles de dimension~$3$ uniréglées et rationnellement connexes (Théorème~\ref{thm.tout}), qui est l'objectif principal de cet article. La section~5 est dédiée à la preuve d'un équivalent du Théorème de Comessatti pour les surfaces singulières de Du Val (Défintion~\ref{dfn.duval}) qui a servi à la preuve des résultats sur les variétés rationnellement connexes du Théorème~\ref{thm.tout}. Quelques questions et conjectures ont été regroupées en section~\ref{sec.conj} et nous avons ajouté deux annexes pour le confort du lecteur, l'une concernant les éclatements, l'autre les conjectures de Thurston.
Le théorème de Viterbo ayant fait l'objet d'un séminaire Bourbaki en 2000 \cite{KhBourbaki}. Nous allons nous concentrer sur les résultats obtenus depuis lors, dont le plus récent a été publié en 2012 \cite{mw1}. 

Pour conclure cette introduction, remarquons que si $M$ est n'est pas lisse, il n'existe pas, en général,  d'ensemble algébrique réel homéomorphe à $M$. Akbulut et King ont néanmoins prouvé le résultat suivant qui généralise le théorème de Nash. Toute espace topologique qui porte une structure de \emph{variété $PL$}\footnote{Piecewise Linear, cf. e.g. \cite{RS82}.} compacte sans bord est homéomorphe à un ensemble algébrique réel \cite{AK81}.

\bigskip
\nocite{Si89,DIK}
Je me suis largement inspiré de \cite{Ko01real} pour rédiger la première partie de cet article.
Je tiens à remercier pour leurs relectures, corrections, améliorations et soutien (par ordre alphabétique):
Alexandre Bardet, Mohamed Benzerga, Jérémy Blanc, Sophie Blanc, Anne-Françoise Coïc, Luck Darnière, Sorin Dumitrescu, Michel Granger, Ilia Itenberg, Tan Lei, Jean-Jacques Loeb, Gustave Mangolte, Jeanne Mangolte, Luc Menichi, Jean-Philippe Monnier, Daniel Naie, Geoffrey Powell. Les erreurs résiduelles m'incombant comme il se doit.

%%%%%%%%%%%%%%%%%%%%%%%%%
\section{La conjecture de Nash de 1952 à 2000 en passant par 1914}\label{sec.conjnash}

\subsubsection*{Théorèmes de Nash et Tognoli}

Nous suivrons les conventions suivantes pour l'usage de l'expression \emph{variété algébrique réelle}. Une variété algébrique $X$ est une variété (quasi-)projective, c'est-à-dire (un ouvert d')une variété définie par des équations polynomiales (homogènes dans le cas projectif). Une telle variété algébrique est \emph{réelle} si elle admet une réalisation dont les équations sont à coefficients réels. Pour fixer les idées, supposons que $X$ soit définie  
par des polynômes homogènes à coefficients réels
$$
P_1(z_0,\dots ,z_m),\dots,P_r(z_0,\dots ,z_m)
$$
ces polynômes déterminent un sous-espace topologique de $\PP^m(\CC)$
\begin{multline*}
X:=X(\CC)=\{(z_0:\dots :z_m)\in\PP^m(\CC) \st\\ 
P_1(z_0,\dots ,z_m)=\dots=P_r(z_0,\dots ,z_m)=0\}
\end{multline*} 
et un sous-espace topologique de $\PP^m(\RR)$ 
\begin{multline*}
X(\RR):= \{(z_0:\dots :z_m)\in\PP^m(\RR) \st\\ 
P_1(z_0,\dots ,z_m)=\dots=P_r(z_0,\dots ,z_m)=0\}
\end{multline*} 
appelé le \emph{lieu réel} (ou parfois la \emph{partie réelle}) de $X$.\label{p.lieu}

Par exemple, le lieu réel du plan projectif complexe est le plan projectif réel $\PP^2(\RR)=\RR\PP^2$. Le lieu réel de la surface quadrique de $\CC^3_{x,y,z}\subset\PP^3(\CC)$ d'équation $x^2+y^2-z^2=0$ est un cône de révolution. Le lieu réel de $\PP^1(\CC)\times\PP^1(\CC)\approx \sS^2\times \sS^2$ est le tore $\PP^1(\RR)\times\PP^1(\RR)\approx \sS^1\times \sS^1$.

Une \emph{surface} algébrique réelle $X$ est donc une variété algébrique réelle telle que $\dim_\RR X(\CC)=4$, et plus généralement, une variété algébrique réelle $X$ de dimension $n$ vérifie $\dim_\RR X(\CC)=2n$.

A priori, le lieu réel peut être vide (par exemple $x^2+y^2+z^2=-1$ dans $\RR^3$), mais lorsque la variété $X$ possède au moins un point réel non singulier\footnote{Dans ce cas, en chaque point non singulier $x$, $X$ possède une dimension en tant que variété et on définit la dimension de $X$ comme le max des dimensions obtenues.}, on a $\dim_\RR X(\CC)=2\dim_\RR X(\RR)$.
En particulier, lorsque la variété algébrique est non singulière et $X(\RR)\ne\emptyset$, les sous-ensembles algébriques $X(\CC)\subset\PP^m(\CC)$ et $X(\RR)\subset\PP^m(\RR)$ sont munis chacun d'une structure de sous-variété compacte de classe~$\mathcal{C}^\infty$.  

Réciproquement étant donné une variété $\cC^\infty$, peut-on la considérer comme lieu des points d'une variété algébrique lisse?
Il est tout à fait clair qu'en général, une variété $\cC^\infty$ n'est difféomorphe à aucune variété algébrique \emph{complexe}. En effet la variété différentielle sous-jacente à une variété complexe lisse est orientable et de dimension paire. Par ailleurs bien d'autres obstructions plus fines sont connues cf. e.g. \cite{FM94} pour des résultats modernes. En revanche, Nash a prouvé qu'il n'y a pas d'obstruction en réel dans le cas compact. 

\begin{thm}[Nash 1952]\label{thm.nash}
Si $M$ est une variété $\cC^\infty$ compacte connexe sans bord, alors il existe une variété algébrique projective réelle $X$ dont une composante connexe $A\subset X(\RR)$ du lieu réel est difféomorphe à~$M$,
$$
M\approx A\hookrightarrow X(\RR).
$$
\end{thm}

On pourra lire une preuve de ce théorème dans \cite{Nash52} ou \cite[Théorème 14.1.8]{BCR}.

À la suite de son théorème, Nash propose deux conjectures qui en renforcent la conclusion. 
La première de ces conjectures affirme qu'il existe une variété $X$ telle que $X(\RR)\approx M$. Cette conjecture  a été prouvée par A.~Tognoli au début des années 70.

\begin{thm}[Tognoli 1973]
On peut, dans l'énoncé de Nash, imposer $X(\RR)$ connexe.
\end{thm}

La preuve (\cite{To73} ou \cite[Théorème 14.1.10]{BCR}) utilise un résultat profond de la théorie du cobordisme qui affirme que toute variété $\cC^\infty$ compacte est cobordante à un ensemble algébrique réel compact non singulier.

Il est alors aisé de construire une variété algébrique réelle dont le lieu réel est réunion des lieux réels de variétés données a priori.

\begin{cor}[Théorème de Nash-Tognoli]
Si $M$ est une variété $\cC^\infty$ compacte sans bord, alors il existe une variété algébrique projective réelle $X$ dont le lieu réel est difféomorphe à~$M$:
$$
M\approx X(\RR).
$$
\end{cor}

\subsubsection*{Variétés rationnelles} 

La seconde conjecture de Nash est l'énoncé suivant:
\begin{conj}[Nash 1952]\label{conj.nash}
Soit $M$ est une variété $\cC^\infty$ compacte connexe sans bord, alors il existe une variété $X$ \emph{rationnelle} dont le lieu réel est difféomorphe à~$M$.
\end{conj}

Cette conjecture avait été contredite par anticipation en 1914 pour les surfaces projectives lisses, puis prouvée au début des années 90 pour les variétés projectives de dimension $3$ singulières, puis contredite à la fin des années 90 pour les variétés de dimension $3$ et plus, lisses et projectives, et enfin prouvée pour les variétés de dimension $3$ lisses mais compactes non-projectives ! 
Nous allons détailler ces résultats portant sur les différentes propriétés des modèles algébriques dans la suite de cette section.

\begin{dfn}\label{p.ratio}\label{dfn.ratio}
Une variété $X$ de dimension $n$ sur un corps $K$ est \emph{rationnelle} si et seulement si elle est birationnellement équivalente à l'espace projectif $\PP_K^n$. C'est-à-dire s'il existe des ouverts de Zariski denses $U\subset X$, $V\subset \PP_K^n$ et un isomorphisme $U\stackrel{\cong}{\longrightarrow} V$ défini par des quotients de polynômes à coefficients dans $K$. 
\end{dfn}

\begin{rems*}
\begin{enumerate}
\item L'éclatement d'une variété le long d'une sous-variété (revoir construction en annexe~\ref{anex.eclat}) est un morphisme birationnel.
\item Dans la conjecture~\ref{conj.nash} les variétés sont rationnelles sur $K=\RR$.
\item Les variétés $\PP^n(K)$ et $K^n$ sont rationnelles sur $K$.
\item La surface $\PP^1\times\PP^1$, les surfaces de Hirzebruch $\FF_k$ (fibrés en $\PP^1$ sur $\PP^1$), les fibrés en coniques (e.g. $x^2+y^2=P(z)$ pour $P\in\RR[z]$) sont des exemples de surfaces rationnelles sur $\RR$.

\end{enumerate}
\end{rems*}

Il est bien connu que lorsqu'une variété algébrique est irréductible, son anneau des fonctions rationnelles est un corps appelé  \emph{corps des fonctions} de la variété. On sait que le corps des fonctions d'une variété algébrique (intègre) de dimension $n$ sur $K$ est une extension de degré fini d'un corps de fractions rationnelles à $n$ indéterminées $K(X_1,\dots,X_n)$. La variété $X$ est alors rationnelle  si et seulement si son corps des fonctions est \emph{isomorphe} à  $K(X_1,\dots,X_n)$.

La Conjecture \ref{conj.nash} est beaucoup plus forte que le Théorème~\ref{thm.nash} puisqu'elle affirme que l'on peut choisir $X$ avec un corps des fonctions de degré~$1$ sur $\RR(X_1,\dots,X_n)$ quel que soit $M$.

Observons pour terminer ce paragraphe qu'une variété algébrique réelle est une variété dont les équations polynomiales sont à coefficients réels, alors qu'une variété rationnelle n'a rien à voir avec le fait que ses équations soient à coefficients rationnels.

\subsection*{Nash pour les surfaces}

La conjecture de Nash était déjà contredite en dimension~$2$ par un théorème de Comessatti de 1914 \cite{Co14}, article peu connu à l'époque de Nash.

\begin{thm}[Théorème de Comessatti]\label{thm.co} 
Soit $X$ une surface projective réelle non singulière. Si $X$ est rationnelle,
alors son lieu réel $X(\RR)$ est difféomorphe à  $\sS^2$, à $\sS^1\times\sS^1$, ou à une surface non orientable. Réciproquement, chacune de ces surfaces topologique admet un modèle algébrique rationnel lisse. 
\end{thm}

En suivant \cite{BM92}, remarquons que l'on obtient un résultat différent si l'on permet des singularités et que l'on remplace "difféomorphe" par "homéomorphe". Pour construire un modèle rationnel pour chaque surface topologique, on commence par éclater $k$ points du plan projectif $\RR\PP^2$ pour obtenir une surface algébrique $X_k$ telle que $X_k(\RR)$ est non orientable de caractéristique d'Euler $1-k$. On obtient ainsi un modèle rationnel lisse de n'importe quelle surface non orientable. Si les $k=2g>0$ points sont choisis alignés sur une droite $H$, on peut contracter la transformée birationnelle $\widetilde H\subset X_{2g}$ pour obtenir une surface algébrique $Y_g$. La surface $Y_1$ est lisse et $Y_1(\RR)$ est difféomorphe au tore, mais si $g>1$, la surface $Y_g$ est singulière au point $P$ (l'image de $\widetilde H$ par la contraction) et $Y_g(\RR)$ est seulement homéomorphe à une surface orientable de genre $g$. 

En effet, comme expliqué dans l'annexe~\ref{anex.eclat}, l'éclatement topologique d'une surface orientable $S_g$ de genre $g$ centré en un point $Q\in S_g$ est difféomorphe à une somme connexe de $2g+1$ plans projectifs $B_QS_g\approx S_g\#\RR\PP^2\approx\RR\PP^2\#\dots\#\RR\PP^2$. En particulier, $S_g\setminus\{Q\}\approx X_{2g}(\RR)\setminus H\approx Y_g(\RR)\setminus \{P\}$, la surface $Y_g(\RR)$ est donc homéomorphe à $S_g$. Et si $g>1$, $Y_g$ ne peut être lisse d'après le théorème de Comessatti. 

Au vu de ces nuances, il est naturel de proposer des variations sur le thème de la conjecture de Nash. Nous dressons un petit panorama des versions envisagées par différents auteurs.

\subsection*{Nash topologique est vraie}

On renvoie à l'annexe~\ref{anex.eclat} pour la définition des éclatements et contractions dans le cadre différentiable. Le résultat suivant, qui peut être vu comme un analogue topologique de la conjecture de Nash, a été prouvé en dimension $3$ par Akbulut et King \cite{AK91} et par Benedetti et Marin \cite{BM92}, puis en toute dimension par Mikhalkin \cite{Mi97}. 

\begin{thm}[\cite{Mi97}]
Toute variété $\cC^\infty$ compacte connexe est difféomorphe à une variété $\cC^\infty$ obtenue à partir de $\RR\PP^n$ par une suite d'éclatements et de contractions différentiables.
\end{thm}

\subsection*{Nash projective singulière est vraie si $n\leq3$}

Nous avons expliqué précédemment la preuve de cette version en dimension $2$. Pour les variétés de dimension $3$, il ne suffit plus d'éclater des points et de contracter des diviseurs, il faut aussi autoriser certaines chirurgies le long de n{\oe}uds. Rappelons que topologiquement, toute $3$-variété s'obtient à partir de la sphère $\sS^3$ par chirurgie le long d'un n{\oe}ud. Une \emph{chirurgie} le long d'un \emph{n{\oe}ud}\footnote{C'est-à-dire que $L$ est un cercle plongé dans $M$, on se restreint ici pour simplifier aux n{\oe}uds dont un voisinage tubulaire dans $M$ est orientable. Un voisinage tubulaire fermé de $L$ est alors difféomorphe à $\sS^1\times \DD^2$.} $L$ dans une variété $M$ consiste à recoller un tore solide $T:=\sS^1\times \DD^2$ au bord du complémentaire d'un voisinage tubulaire ouvert $U_L$ de $L$. Ce recollement est réalisé par un difféomorphisme $\varphi\in\Diff(\sS^1\times\sS^1)$ du tore $\sS^1\times \sS^1=\partial \left(M\setminus U_L\right)=\partial T$ sur lui-même. L'opération qui produit $M_\varphi=M\setminus U_L\cup_\varphi T$ à partir de $M$ s'appelle une \emph{chirurgie} le long de $L$. Benedetti et Marin montrent qu'à l'exception de certains types topologiques traités à part, la plupart des variétés de dimension $3$ sont obtenues à partir de $\sS^3$ par éclatements de points et certaines chirurgies qu'ils appellent \emph{déchirures}\label{p.dechi}. Cette présentation des transformations topologiques leur permet de montrer la conjecture de Nash topologique. Ils réalisent alors de manière algébrique les déchirures et obtiennent une variété $Y$ singulière\footnote{Ici "singulière"="éventuellement singulière".} et une résolution des singularités $X\to Y$ telles que $X$ est lisse et birationnelle à $\sS^3$, et $Y(\RR)$ est homéomorphe à $M$.

\begin{thm}[\cite{BM92}]
Soit $M$ une variété $\cC^\infty$ compacte connexe de dimension~3. Alors il existe une variété algébrique projective réelle singulière $X$ rationnelle et telle que $X(\RR)$ soit homéomorphe à $M$.
\end{thm}

\subsection*{Nash non projective non singulière est vraie si $n=3$}

Une variété projective non singulière est en particulier une variété analytique complexe compacte. Réciproquement, si l'on suppose en plus que le corps des fonctions méromorphes d'une variété analytique complexe compacte est de degré de transcendance maximal (c'est-à-dire égal à la dimension), on obtient une variété qui est très proche d'une variété projective. Pourtant, la conjecture de Nash est satisfaite pour ces variétés en dimension~$3$ alors qu'elle est fausse pour les variétés projectives comme on le verra avec le Théorème~\ref{thm.kollar}.

\begin{dfn}
Une variété analytique complexe non singulière compacte de dimension $n$ est \emph{de Moishezon} si elle possède $n$ fonctions méromorphes algébriquement indépendantes ou de façon équivalente si elle est biméromorphe à une variété projective.\footnote{À mettre en regard avec la discussion qui suit \ref{dfn.ratio}.}
\end{dfn}

Toute surface de Moishezon est projective \cite[IV.5]{BPV}. Les premiers exemples de variétés de Moishezon non projectives en dimension 3 sont dûs à Hironaka, cf. \cite[App.B.3]{Ha77}.

On a vu qu'une variété algébrique est réelle lorsqu'elle est définie par des polynômes à coefficients réels. Nous constatons qu'une variété projective
\begin{multline*}
X=\{(z_0:\dots :z_m)\in\PP^m(\CC) \st\\ 
P_1(z_0:\dots :z_m)=\dots = P_r(z_0:\dots :z_m)=0\}
\end{multline*}  
est réelle si et seulement si $X$ est globalement fixée par la restriction de le conjugaison complexe 
$$
\sigma_0 \colon \PP^m(\CC) \to \PP^m(\CC),\  (z_0:\dots:z_m)\mapsto (\bar z_0:\dots:\bar z_m).
$$

En suivant \cite{Ko-nonproj}, on dira qu'une variété de Moishezon est réelle si elle admet une action de Galois:

\begin{dfn}
Une \emph{variété de Moishezon réelle} est une variété de Moishezon munie d'une involution anti-holomorphe globale $\sigma\colon X\to X$.
\end{dfn}

Par analogie avec le cas projectif, on note alors $X(\RR):=X^\sigma$ le lieu fixe de~$\sigma$.

\begin{thm}[\cite{Ko-nonproj}]
Soit $M$ une variété $\cC^\infty$ de dimension~3 compacte et connexe. Alors il existe une variété de Moishezon réelle $(X,\sigma)$, et une application biméromorphe $\pi\colon \PP^3\dashrightarrow X$ vérifiant $\pi \sigma_0=\sigma\pi$ et telle que $X(\RR)$ soit difféomorphe à $M$.
\end{thm}

On peut être plus précis (voir \cite{Ko-nonproj}) et dire qu'il existe une suite d'éclatements et de contractions à centres lisses (Voir Appendice~\ref{anex.eclat})
$$
\PP^3=Y_0\stackrel{\pi_0}{\dashrightarrow} Y_1 \stackrel{\pi_1}{\dashrightarrow} \cdots \stackrel{\pi_{n-1}}{\dashrightarrow} Y_n=X
$$
où pour tout $i$, la variété $Y_i$ est non singulière. De plus cette suite est réelle dans le sens suivant:  chaque variété est munie d'une involution anti-holomorphe globale $\sigma_i\colon Y_i\to Y_i$; $\sigma_n=\sigma$; et ces structures réelles vérifient $\pi_i\sigma_i=\sigma_{i+1}\pi_i$ pour tout $i$.

Pour parvenir à ce résultat,  Koll\'ar utilise la classification de Benedetti-Marin de ce qu'il appelle les "flops topologiques" qui sont un cas particulier des déchirures de la page~\pageref{p.dechi}. Ensuite il montre comment les réaliser par des flops algébriques. Décrivons rapidement le type particulier de \emph{flop algébrique} utilisé par Koll\'ar. Il s'agit d'une application birationnelle $f\colon X\dasharrow X'$ qui se factorise $X\stackrel{\pi}{\longleftarrow} X_1\stackrel{\pi'}{\longrightarrow} X'$ où $\pi$ et $\pi'$ ont un même diviseur exceptionnel $E\subset X_1$ qui est isomorphe à $\PP^1\times \PP^1$, chaque morphisme contractant un des deux facteurs $\PP^1$. % et $f=\pi^{-1}\pi'$. 
La transformation du lieu réel $X(\RR)\dasharrow X'(\RR)$ est alors un flop topologique. 
Réciproquement, l'existence d'une telle transformation sur une variété $X$ de dimension $3$ nécessite une courbe rationnelle $C\subset X$ plongée d'une façon bien particulière :  
\begin{enumerate}
\item le diviseur exceptionnel $E$ de l'éclatement $\pi\colon X_1\to X$  centré en $C$ est isomorphe à $\PP^1\times\PP^1$, 
\item $\pi_{\vert E}\colon E \to \PP^1$ est la projection sur le premier facteur,
\item il existe une contraction  $\pi'\colon X_1\to X'$ de $E$ dont la restriction à $E$ est la projection sur le second facteur. 
\end{enumerate}

La première étape pour prouver le théorème est de réaliser une approximation algébrique convenable de certains rubans de M\"obius plongés représentant des flops topologiques. La seconde étape est obtenue en construisant des flops algébriques grâce à des éclatements qui ne modifient pas le lieu réel. Dans ce processus, la variété $X'$ n'est plus projective en général mais reste de Moishezon puisque le corps des fonctions est préservé par transformation birationnelle. On trouvera la construction complète dans \cite[\S 4]{Ko01real}.

%\newpage
\subsection*{Nash projective non singulière est fausse pour $n>1$}

\subsubsection*{Théorème de Koll\'ar}

L'échec de la conjecture de Nash en dimension~$3$ est une conséquence du théorème ci-dessous, prouvé dans la série d'articles \cite{KoI,KoII,KoIII,KoIV}.
L'énoncé s'applique à une classe de variétés généralisant les variétés rationnelles. 

\begin{dfn}\label{dfn.uni} Une variété réelle ou complexe $X$ de dimension $n$ est \emph{uniréglée} si elle est dominée par un cylindre de même dimension. C'est-à-dire s'il existe une variété $Y$ de dimension $n-1$ et une application rationnelle 
$$
Y\times\PP^1 \dasharrow X
$$
d'image dense pour la topologie de Zariski\footnote{La définition est exactement la même que $X$ soit complexe ou réelle; "uniréglée" est une propriété invariante par changement de base.}. 
\end{dfn}

La variété produit $\PP^{n-1}\times\PP^1$ étant birationnellement équivalente à $\PP^n$, il est  immédiat qu'une variété rationnelle sur $\RR$ ou sur $\CC$ est uniréglée.

Nous rappellerons plus loin les définitions topologiques apparaissant dans le théorème, pour l'instant il suffit de savoir qu'il s'agit de variétés très particulières et que la liste donnée est très loin de recouvrir toutes les variétés de dimension $3$.

\begin{thm}[Koll\'ar  1998~\hbox{\cite[Th.~6.6]{Ko01}}]\label{thm.kollar}
Soit $X$ une variété algébrique projective réelle non singulière de dimension~$3$. Supposons  $X$ uniréglée et  $\xr$ orientable, alors toute composante connexe de $\xr$ est difféomorphe à l'une des variétés suivantes:
\begin{enumerate}
\item une variété de Seifert,
\item une somme connexe d'un nombre fini d'espaces lenticulaires,
\item\label{kollar.fibrestores} un fibré  localement trivial en tores $\sS^1\times\sS^1$ au-dessus de $\sS^1$, ou doublement recouvert par un tel fibré,
\item\label{kollar.except} une variété appartenant à une liste finie d'exceptions,
\item une variété obtenue à partir de l'une des précédentes en effectuant la somme connexe avec un nombre fini de copies de $\RR\PP^3$ et un nombre fini de copies de $\sS^1\times \sS^2$.
\end{enumerate}
\end{thm}

À toutes fins utiles, voici trois définitions topologiques utilisées dans l'énoncé ci-dessus.

\subsubsection*{Somme connexe}\label{somme.connexe}
Soient $M_1$ et $M_2$ deux variétés connexes orientées de même dimension $n$. Soient $B_1\subset M_1$ et $B_2\subset M_2$ des boules ouvertes, les complémentaires $F_1:=M_1\setminus B_1$ et $F_2:=M_2\setminus B_2$ sont des variétés dont le bord est homéomorphe à une sphère $\sS^{n-1}$. En identifiant $F_1$ et $F_2$ le long des sphères qui les bordent par un difféomorphisme compatible avec les orientations induites, on obtient une variété sans bord orientée qui est uniquement déterminée par les variétés orientées $M_1$ et $M_2$ à homéomorphisme près, c'est la \emph{somme connexe} $M_1\#M_2$. Si on note $-M_2$ la variété $M_2$ munie de son orientation opposée, les sommes connexes $M_1\#M_2$ et $M_1\#-M_2$ ne sont pas homéomorphes en général. Signalons néanmoins que les sommes connexes $M_1\#M_2$ et $M_1\#-M_2$ sont homéomorphes lorsque $M_2=\RR\PP^3$ ou $M_2=\sS^1\times \sS^2$, cf. \cite{He76}.

\subsubsection*{Variétés de Seifert}\label{p.var.seifert}

Soit $\sS^1\times \DD^2$ le tore solide où $\sS^1$ est le cercle unité $\{u\in\CC\st \vert u\vert=1\}$ et $\DD^2$ est le disque unité fermé $\{z\in\CC,\ \vert z\vert\leq1\}$. Une \emph{fibration de Seifert}  du tore solide est une application différentiable de la forme 
$$
f\colon \sS^1\times \DD^2 \to \DD^2,\ (u,z)\mapsto u^qz^p,
$$
où $p,q$ sont des entiers naturels, tels que $p\ne 0$ et $(p,q)=1$.\label{p.solid} L'application $f$ est une fibration en cercles qui est localement triviale au dessus du disque épointé $\DD^2\setminus \{0\}$.

Une variété compacte sans bord $M$ de dimension $3$ est dite \emph{de Seifert} %cf. HDR Maillot
si elle admet une application différentiable $g\colon M \to B$ au-dessus d'une surface $B$ telle que chaque point $P\in B$ admet un voisinage fermé $U$ au-dessus duquel la restriction de $g$ à $g^{-1}(U)$ est difféomorphe à une fibration de Seifert du tore solide. En particulier, toute fibre de $g$ est difféomorphe à $\sS^1$ et  $g$ est localement triviale en dehors d'un ensemble fini $\{P_1,\dots,P_k\}\subset B$; la fibre $g^{-1}(P_i)$ étant multiple.

\subsubsection*{Espaces lenticulaires}

Soit $n\in \NN^*$, notons $\mu_n$ le sous-groupe multiplicatif de $\CC^*$ des racines $n$-ièmes\ de l'unité. Soit $0<q<p$ des entiers premiers entre eux. L'\emph{espace lenticulaire} $\LL_{p,q}$ est le quotient de la sphère 
$$
\sS^3=\{(w,z)\in\CC^2\st\vert w\vert^2+\vert z\vert^2=1\}
$$ par l'action de $\mu_p$ définie par 
$$
\zeta\cdot(w,z)=(\zeta w,\zeta^qz),
$$
pour tout $\zeta\in\mu_p$ et tout $(w,z)\in\CC^2$.

Observons que tout espace lenticulaire admet une fibration de Seifert (en fait un tel espace admet une infinité de fibrations de Seifert). Les espaces du type $\LL_{p,1}$  admettent même une fibration localement triviale. Une telle fibration s'obtient par exemple à partir de la fibration de Hopf de la sphère $\sS^3$ au-dessus de la sphère $\sS^2\approx\CC\cup\{\infty\}$. 
$$
\begin{array}{lclc}
 &\sS^3 &\longrightarrow& \sS^2 \\
&(w,z)  &\longmapsto &w/z    
\end{array}
$$

Un quotient cyclique de la fibration de Hopf est une fibration de Seifert au-dessus d'un orbifold de dimension $2$ (voir p.~\pageref{p.orbi} pour la notion d'orbifold) :
$$
\begin{array}{lclc}
 &\LL_{p,q} &\longrightarrow& \sS^2(p,q) \\
&(w,z)  &\longmapsto &w^q/z.    
\end{array}
$$

Nous venons de voir que tout espace lenticulaire est une variété de Seifert. Par contre, à l'exception de $\LL_{2,1}\#\LL_{2,1}=\RR\PP^3\#\RR\PP^3$, une somme connexe d'au moins deux espaces lenticulaires n'admet aucune fibration de Seifert (voir à ce sujet \cite[p.~457]{Sc83}).

\subsubsection*{Théorème de Viterbo}

Parmi les variétés topologiques de dimension $3$, les variétés quotient $\raisebox{-.65ex}{\ensuremath{\Lambda}}\!\backslash \HH^3$ où $\Lambda\subset \PO(3,1)$ est un sous-groupe discret de déplacements\footnote{$\HH^3$ est le demi-espace $\{(x,y,z)\in\RR^3\st z>0\}$.}, c'est-à-dire les variétés \emph{hyperboliques}, sont celles dont la géométrie est la plus riche et que l'on connait le moins bien.

Avec son théorème, Koll\'ar a prouvé qu'à un nombre fini d'exceptions près, les variétés hyperboliques de dimension $3$ ne sont pas uniréglées. Il a conjecturé que cette situation était plus générale. Peu de temps après, Viterbo et Eliashberg ont prouvé cette conjecture en montrant qu'en toute dimension supérieure à deux, les variétés hyperboliques ne sont pas uniréglées, \cite{V99}, \cite[1.7.5]{EGH}.

\begin{thm}\label{thm.vit}
Soit $W$ une variété projective non singulière de dimension complexe $>2$, et $L\subset W$ une sous-variété $\cC^\infty$ plongée qui est lagrangienne pour la structure symplectique sous-jacente à $X(\CC)$. Si $W$ est uniréglée, alors $L$ n'admets pas de métrique riemannienne à courbure sectionnelle strictement négative.
\end{thm}

C'est un exercice classique de montrer que le lieu réel $X(\RR)$ d'une variété projective non singulière réelle $X$ est une sous-variété lagrangienne de la variété symplectique sous-jacente à $X(\CC)$ obtenue en considérant la structure de variété kählerienne sur $X(\CC)$ induite par celle de l'espace projectif ambiant.

\begin{cor}\label{cor.vit}
Soit $X$ une variété projective non singulière réelle de dimension $>2$. Si $X$ est uniréglée, alors aucune composante connexe de $X(\RR)$ n'admet de métrique hyperbolique\footnote{Une telle métrique possède une courbure sectionnelle constante égale à $-1$.}.
\end{cor}

À la suite de son théorème, Koll\'ar a posé plusieurs questions et proposé plusieurs conjectures concernant la classification des variétés uniréglées réelles, nous y reviendrons en détails en Section~\ref{var.3}.

%%%%%%%%%%%%%%%%%%%%%%%%%
\section{Surfaces rationnelles et bretzels}\label{sec.surf}

Rappelons rapidement la classification topologique des surfaces rationnelles réelles et signalons l'article d'exposition récent \cite{Hu11} sur ces même surfaces.% et le survey à venir sur la géométrie birationnelle de ces surfaces \cite{blm2}.

Une surface obtenue à partir du plan projectif par des éclatements et des contractions dont les centres ne sont pas forcément définis sur $\RR$ est dite \emph{géométriquement rationnelle} ou \emph{rationnelle sur $\CC$}. Si l'on impose que les centres des éclatements et des contractions sont réels\footnote{Ici "réel"="globalement réel", c'est-à-dire que si $E$ est dans le centre d'un éclatement, $\overline{E}$ aussi.}, $X$ est alors rationnelle sur~$\RR$, c'est-à-dire que son corps des fonctions est isomorphe au corps des fractions rationnelles $\RR(X_1,X_2)$ (cf. Définition~\ref{p.ratio} et commentaires).

Précisons la notion d'éclatement décrite en détails dans l'appendice~\ref{anex.eclat} dans le cas d'un point du plan projectif.
L'éclatement $B_{(0:0:1)}\PP^2$ de $\PP^2$ centré en $P=(0:0:1)$ est la surface algébrique $\widetilde\PP^2$ définie localement au-dessus du voisinage $U=(z\ne 0)$ de $P$ par 
$$
B_PU:=\{((x,y),[u:v])\in U_{x,y}\times \PP^1_{u:v} \st uy=vx\}.
$$

Plus généralement, l'éclaté du plan projectif $\PP^2_{x:y:z}$ en un point $P=(a:b:1)$ de l'ouvert affine $(z\ne 0)$ est donné par
\begin{multline*}
B_{(a:b:1)}\PP^2:=\{([x:y:z],[u:v])\in \PP^2_{x:y:z}\times \PP^1_{u:v} \st \\ u(y-bz)-v(x-az)=0\},
\end{multline*}
et en particulier
\begin{multline*}
B_{(0:0:1)}\PP^2:=\{([x:y:z],[u:v])\in \PP^2_{x:y:z}\times \PP^1_{u:v} \st \\ uy-vx=0\}.
\end{multline*}

\begin{prop}
Soit $P\in X$ un point non singulier d'une surface algébrique et $\pi_P\colon B_PX\to X$ l'éclatement de $X$ centré en $P$. La courbe $E_P:=\pi^{-1}\{P\}$ est la \emph{courbe exceptionnelle} de l'éclatement et la restriction de $\pi_P$ à $B_PX\setminus E_P\to X\setminus \{P\}$ est un isomorphisme.
\end{prop}

Si $P$ est réel,  la transformation topologique du lieu réel correspond à la chirurgie suivante : on retire à $X(\RR)$ un disque centré en $P$ (dont le bord est un cercle) et on le remplace par un ruban de M\"obius (dont le bord est aussi un cercle) pour obtenir $B_PX(\RR)$. Ce qui donne:
$$
B_PX(\RR)=X(\RR)\#\RR\PP^2.
$$

\begin{defi}
On  appelle \emph{contraction} l'opération réciproque de l'éclatement.  
\end{defi}

Bien entendu, on ne peut pas contracter n'importe quelle courbe vers un point lisse alors qu'il est possible de centrer un éclatement en n'importe quel point lisse. Pour les surfaces nous avons le critère suivant:

\begin{thm}[Critère de Castelnuovo]
Soit $Y$ une surface projective et $E\subset Y$ une courbe isomorphe à $\PP^1$ qui vérifie\footnote{Sur le $\ZZ$-module libre engendré par les courbes d'une surface projective, il existe une forme bilinéaire symétrique non dégénérée : la \emph{forme intersection}, cf. e.g. \cite[Chapter~V]{Ha77}.} $E\cdot E=-1$,
alors il existe une surface projective $X$ et un morphisme $\pi\colon Y\to X$ tel que $P=\pi(E)$ soit un point lisse de $X$ et $\pi$ soit l'éclatement de $X$ centré en $P$. 
\end{thm}

Pour une preuve, cf. e.g.~\cite[V.5]{Ha77}.
Plus généralement, on peut contracter des courbes vers des points qui ne sont pas forcément lisses. Soit $E\subset Y$ une courbe projective, connexe et réduite et sur une surface projective lisse $Y$ et $E=\sqcup E_i$ sa décomposition en composantes irréductibles.

\begin{thm}[Grauert]
Il existe une surface projective $X$ et un morphisme birationnel $\pi\colon Y\to X$ tel que $P=\pi(E)$ soit un point de $X$ et la restriction de $\pi$ à $Y\setminus E\to X\setminus P$ soit un isomorphisme si et seulement si la matrice $(E_i\cdot E_j)_{i,j}$ est définie négative.
\end{thm}
Cf. e.g.~\cite{BPV}.

De nombreux auteurs appellent \emph{bretzel} à $g$ trous une surface topologique orientable de genre $g$. Le résultat central (dont on a déjà vu un corollaire en \ref{thm.co}) pour les surfaces réelles géométriquement rationnelles  est que les bretzels autorisés n'ont qu'un seul trou.

\begin{thm}\cite{Co14}
Soit $X$ une surface algébrique projective réelle non singulière obtenue à partir du plan projectif $\PP^2$ par un nombre fini d'éclatements et de contractions de centres lisses (réels ou non réels).
Toute composante connexe orientable du lieu réel $X(\RR)$ est difféomorphe à une sphère $\sS^2$ ou à un tore $\sS^1\times\sS^1$.
\end{thm}

L'énoncé précédent est plus général que \ref{thm.co} car nous n'avons pas imposé que les centres des éclatements soient réels. 
La surface $X$ est en fait birationnelle à $\PP^2$ par une application dont les composantes sont à coefficients complexes mais pas forcément réels, c'est une surface \emph{géométriquement rationnelle}.\label{geom.rat}
Plus généralement, nous énonçons ci-dessous la classification pour d'autres surfaces "proches" (voir définitions et discussion p.~\pageref{var.uni}) des surfaces rationnelles. 

\begin{nota} 
Pour la description des types topologiques, nous utiliserons les conventions suivantes:
La sphère de dimension $n$ est notée $\sS^n$, le tore $\sS^1\times \sS^1$, le plan projectif réel $\RR\PP^2$. Pour deux variétés lisses de même dimension $A$ et $B$,
 $A\sqcup B$ est leur réunion disjointe, $A \# B$ leur somme connexe\footnote{En dimension $2$, la somme connexe, définie a priori dans la catégorie des surfaces orientées, ne dépend pas des orientations choisie, on définit donc la somme connexe dans la catégorie des surfaces.}, $\sqcup^sA$ la réunion disjointe de $s$ copies de $A$ et $\#^kA$ la somme connexe de $k$ copies de $A$. Par exemple, $\#^{g}\RR\PP^2$ est la surface non orientable de caractéristique d'Euler $2-g$. Toujours par convention, $\sqcup^0A=\#^0A=\emptyset$.
\end{nota}

\begin{thm}[Classification]
Soit $X$ une surface algébrique projective réelle non singulière.
\begin{enumerate}
\item Si $X$ est rationnelle,
alors $X(\RR)$ est difféomorphe à l'une des surfaces suivantes~:
\begin{enumerate}
\item $X(\RR)\approx \sS^1\times\sS^1$;
\item $X(\RR)\approx \sS^2$;
\item  $\#^{g}\RR\PP^2$ pour $g\in\NN$.
\end{enumerate}

\item Si $X$ est géométriquement rationnelle, alors $X(\RR)$ est difféomorphe à l'une des surfaces suivantes~:
\begin{enumerate}
\item $X(\RR)\approx \sS^1\times\sS^1$;
\item $X(\RR)\approx \sqcup^s\sS^2\sqcup  \#^{g_1}\RR\PP^2\sqcup\dots \sqcup  \#^{g_l}\RR\PP^2$ 

\hfill{pour $s,l,g_1,\dots,g_l\in\NN$.}
\end{enumerate}

\item Si $X$ est uniréglée, alors il existe des entiers naturels $t,s,l,g_1,\dots,g_l$ tels que 
$$
X(\RR)\approx \sqcup^t\left(\sS^1\times\sS^1\right) \sqcup^s\sS^2 \sqcup  \#^{g_1}\RR\PP^2\sqcup\dots \sqcup  \#^{g_l}\RR\PP^2.
$$

\item Réciproquement, toute surface topologique appartenant à la liste 1 (respectivement 2, respectivement 3) admet un modèle rationnel, (respectivement géométriquement rationnel, respectivement uniréglé).
\end{enumerate}
\end{thm}

Nous reviendrons aux surfaces rationnelles dans le cas singulier en section~\ref{sec.singsurf}.

%%%%%
\section{Variétés réelles de dimension~3 de 2000 à 2012}\label{sec.uni}

Le Théorème~\ref{thm.kollar} contraint fortement le lieu réel d'une variété uniréglée. Les conjectures que Koll\'ar a proposées à la suite de son théorème décrivent une classification topologique pour ce lieu réel. L'objet de cette partie est de faire le point sur ces conjectures. Nous commençons par rappeler la notion de variété topologique \emph{géométrique} qui nous permettra d'énoncer et de mettre en perspective les résultats présentés.

\subsection*{Variétés topologiques de dimension~3}\label{var.3}

La classification topologique\footnote{identique à la classification $\cC^\infty$ pour les surfaces.} des surfaces connexes compactes sans bord se résume à deux invariants : un invariant binaire, l'orientabilité, et un invariant entier naturel,  la caractéristique d'Euler. Pour les variétés de dimension~$3$, la situation est plus riche. Nous avons décrit en pages~\pageref{p.var.seifert} et \nextpageref{p.var.seifert}
les variété de Seifert et les espaces lenticulaires qui sont deux classes importantes. Nous rappelons dans ce paragraphe la construction des variétés fondamentales pour l'énoncé de la classification complète (classification abordée en Annexe~\ref{anex.dim3}).

\subsubsection*{Variétés $\cC^\infty$ géométriques}

Une variété riemannienne $\Omega$ est \emph{homogène} si son groupe d'isométries $\Isom(\Omega)$ agit transitivement sur $\Omega$.
Une \emph{géométrie} $\Omega$ est une variété riemannienne simplement connexe, homogène qui admet un quotient de volume fini. 
Si $\Omega$ est un groupe de Lie réel, on en fait une variété riemannienne en le munissant d'une métrique invariante à gauche et on parle de "la" géométrie $\Omega$. %\footnote{La notation $\raisebox{-.65ex}{\ensuremath{\Lambda}}\!\backslash M$ est utilisée pour un espace d'orbites sous une action de groupe}

\begin{dfn}\label{dfn.geom}
Une variété $M$ de classe $\cC^\infty$  est \emph{géométrique} si $M$ est difféomorphe au quotient d'une géométrie $\Omega$ par un sous-groupe discret d'isométries $\Lambda\subset \Isom(\Omega)$ agissant sans point fixe. On dit aussi que $M=\raisebox{-.65ex}{\ensuremath{\Lambda}}\!\backslash \Omega$ \emph{admet une structure géométrique modelée} sur $\Omega$. Par extension, une variété à bord est \emph{géométrique} si son intérieur est géométrique.
\end{dfn}

Lorsque $\Omega$ est un groupe de Lie, les hypothèses retenues impliquent l'existence d'un réseau de covolume fini, c'est-à-dire que $\Omega$ est un groupe de Lie \emph{unimodulaire}.

Une variété de dimension $n$ est \emph{sphérique}, (resp. \emph{euclidienne}, resp. \emph{hyperbolique}) si elle admet une géométrie modelée sur $\sS^n$ (resp. $\EE^n$, resp. $\HH^n$\footnote{$\HH^n$ est le demi-espace $\{(x_1,\dots,x_n)\in\RR^n\st x_n>0\}$ muni de la métrique $\frac1{x_n^2}(dx_1^2+\cdots+dx_{n-1}^2+dx_n^2)$.}).

En dimension $2$, le théorème d'uniformisation induit que toute surface compacte admet une géométrie sphérique, euclidienne ou hyperbolique. Il est remarquable que chacune de ces géométries soit de courbure constante.

Ceci n'est plus le cas en dimension~$3$ où en sus des géométries "à courbure constante" $\sS^3$, $\EE^3$ et $\HH^3$, ou trouve les cinq géométries $\sS^2\times \EE^1,\ \HH^2\times \EE^1,\ \widetilde{\SL_2(\RR)},\ \Nil$ et $\Sol$. Thurston a montré qu'à équivalence près, les huit géométries précédentes sont les seules en dimension~$3$ pourvu que l'on impose à leur groupe d'isométries d'être maximal \cite[page~2]{Be-3mfd-10}.
\label{p.8geom} 

Nous ne décrirons pas ici de manière détaillée les huit géométries, (voir \cite{Sc83}) mais nous nous contentons de définir rapidement certains groupes de Lie. Le groupe $\widetilde{\SL_2(\RR)}$ est le revêtement universel de $\SL_2(\RR)$. Le groupe $\Nil$ est le groupe de Heisenberg des matrices $3\times3$ triangulaires supérieures dont les éléments diagonaux sont tous égaux à $1$. Le groupe $\Sol$ est le seul groupe de Lie de dimension $3$ simplement connexe, admettant un quotient de volume fini qui soit résoluble mais non nilpotent, nous le détaillons en page \pageref{p.sol}.

Nous avons vu que toute surface compacte est géométrique.
A contrario une variété $M$ de dimension~3 n'admet pas  en général de structure géométrique. En revanche, lorsque $M$ admet une structure géométrique, cette dernière est unique pourvu que $M$ soit de volume fini. 

Un résultat important de la théorie est que les variétés de Seifert admettent une structure géométrique, on a même une caractérisation des six géométries "de Seifert" (voir \cite{Sc83} par exemple).\label{p.seifert} 

\begin{prop}\label{prop.seifert}
Une variété $\cC^\infty$ compacte sans bord $M$ de dimension $3$, orientable,
est une variété de Seifert si et seulement si $M$ admet une géométrie modelée sur l'une des six géométries
$$
\sS^3,\ \sS^2\times \EE^1,\ \EE^3,\ \Nil,\ \HH^2\times \EE^1,\ \widetilde{\SL_2(\RR)}.
$$
\end{prop}

Cette proposition reste vraie pour $M$ non orientable si l'on étend la définition de variété de Seifert aux variétés portant un feuilletage en cercles (ce qui revient à admettre un modèle local non orientable en sus des modèles vus en page~\pageref{p.solid}).

\begin{cor}
Soit $M=\raisebox{-.65ex}{\ensuremath{\Lambda}}\!\backslash \Omega$ une variété géométrique, alors $M$ est soit une variété de Seifert, soit $\Omega=\Sol$, soit $\Omega=\HH^3$.
\end{cor}

\subsection*{Variétés géométriques uniréglées}\label{var.uni}
 
Depuis les travaux de Koll\'ar, nous savons que modulo les variétés $\Sol$ et modulo un nombre fini de variétés de dimension $3$ fermées, les variétés réelles uniréglées orientables sont essentiellement des fibrés de Seifert ou des sommes connexes d'espaces lenticulaires. Les progrès accomplis depuis peuvent être résumés rapidement de la manière suivante. Grâce au théorème \ref{thm.vit}, nous savons qu'une variété hyperbolique ne peut être contenue dans le lieu réel d'un variété projective uniréglée non singulière. Le théorème \ref{thm.mw} nous apprend qu'il existe au plus un nombre fini de variétés $\Sol$ qui puissent être contenues dans le lieu réel d'une variété projective uniréglée non singulière. Réciproquement, le théorème~\ref{thm.hm1} nous dit que toute variété géométrique orientable qui n'est ni hyperbolique ni $\Sol$, est difféomorphe à une composante connexe du lieu réel d'une variété projective uniréglée non singulière. 

Nous avons rencontré les variétés unireglées dans l'énoncé du théorème de Koll\'ar (Définition~\ref{dfn.uni}). Nous nous intéressons maintenant à une notion intermédiaire entre rationnel et uniréglé cf.~\cite[Sections~5\&6]{Ko-simplest}.

\begin{dfn}\label{dfn.rc}
Une variété projective réelle $X$ de dimension $n$ est (géométriquement) \emph{rationnellement connexe} (r.~c.) s'il existe un ouvert non vide $U\subset X$ tel que  pour toute paire de points $x,y\in U$, il existe une courbe rationnelle $f\colon \CC\PP^1\to X$ telle que $x,y\in f(\CC\PP^1)$.
\end{dfn}

Par exemple, les hypersurfaces de degré inférieur à $n$ dans $\PP^n$ sont r.~c. et plus généralement toutes les variétés de Fano sont r.~c \cite{KMM92,Cam92}. Pour voir que r.~c. est intermédiaire entre rationnel et uniréglé, remarquons que $X$ est uniréglée si pour tout point $x\in U$ dans un ouvert fixé, il existe une courbe rationnelle $f\colon \CC\PP^1\to X$ telle que $x\in f(\CC\PP^1)$.  En paraphrasant Koll\'ar \cite{Ko-simplest}, on peut dire que la notion de variété r.~c. est actuellement considérée comme la "bonne" généralisation de la notion de variété rationnelle. 

On peut résumer l'état actuel de la classification des variétés uniréglées et rationnellement connexes en dimension $3$ par les énoncés suivants synthétisant des travaux de Catanese, Eliashberg, Givental, Hofer, Huisman, Koll\'ar, Viterbo, Welschinger et moi-même \cite{KoI,KoII,KoIII,KoIV,V99,EGH,hm1,hm2,cm1,cm2,mw1}.
 
\begin{enumerate}
\item Soit $X$ une variété projective réelle non singulière dont le lieu réel est orientable et soit $M\subset \xr$ une composante connexe.
\begin{enumerate}
\item Si $X$ est uniréglée, alors à un nombre fini d'exceptions près, $M$ est essentiellement une variété de Seifert ou une somme connexe d'espaces lenticulaires.
\item Si $M$ admet une fibration de Seifert $M\to B$ avec $B$ orientable, et si $X$ est rationnellement connexe, alors la géométrie de $M$ (cf. \ref{prop.seifert}) n'est ni $\HH^2\times \EE^1$, ni $\widetilde{\SL_2(\RR)}$. 
\end{enumerate}

\item Soit $M$ une variété de Seifert orientable ou une somme connexe d'espaces lenticulaires, alors il existe une variété projective $X$ définie sur $\RR$, uniréglée et non singulière dont le lieu réel $\xr$ contient une composante connexe difféomorphe à $M$.
\end{enumerate}

Pour énoncer ces résultats plus précisément (Théorème~\ref{thm.tout}), 
rappelons que si $M$ est une variété \emph{orientée} compacte sans bord de dimension $3$, alors il existe une décomposition $M = M'
\#^a\RR\PP^3\#^b(\sS^1\times \sS^2)$ avec $a + b$ maximal qui est unique d'après un théorème de Milnor \cite{Mi62} (voir page~\pageref{somme.connexe} pour la définition de somme connexe).

Pour une variété algébrique, être rationnelle, rationnellement connexe ou uniréglée étant invariant par équivalence birationnelle, les propriétés topologiques pertinentes de $M$ dans notre contexte sont en fait portées par $M'$ comme l'illustrent les exemples suivants.

\begin{exem}\cite[Example~1.4]{KoII}\label{ex.blow}

Soit $X$ une variété algébrique réelle de dimension 3, non singulière.
\begin{enumerate}
\item 
Soit $P\in X(\RR)$ un point réel, alors pour la composante connexe $M$ de $X(\RR)$ contenant $P$, nous obtenons (Proposition~\ref{prop.eclat})
$$
B_PM\approx M\#\RR\PP^3.
$$

\item Soit $D\subset X$ une courbe réelle possédant un unique point réel $\{0\}=D(\RR)$. Supposons de plus que proche de $0$, cette courbe soit donnée par les équations $\{z=x^2+y^2=0\}$. Notons $Y_1=B_DX$, la variété obtenue par l'éclatement de $X$ centré en $D$ (Voir Appendice \ref{anex.eclat}), cette nouvelle variété est réelle et possède un unique point singulier $P$. Considérons $Y:=B_PY_1$, la variété obtenue par l'éclatement de $Y_1$ centré en $P$ qui est une variété réelle non singulière. En notant $\pi\colon Y\to X$ la composition des éclatements, nous obtenons pour la composante connexe $M\subset X(\RR)$ contenant $P$  la relation 
$$
\pi^{-1}M\approx M\#(\sS^1\times \sS^2),
$$
c'est-à-dire
$$
B_P(B_DM)\approx M\#(\sS^1\times \sS^2).
$$

\end{enumerate}
\end{exem}

\begin{thm}\label{thm.tout} 
\begin{enumerate}
\item Soit $X$ une variété projective réelle non singulière avec $\xr$ orientable et soit $M\subset \xr$ une composante connexe.
\begin{enumerate}
\item À un nombre fini d'exceptions près, si $X$ est uniréglée alors il existe $a,b\in\NN$ et une variété $M'$ admettant une fibration de Seifert\\ $M'\to B$ ou une décomposition $M'=\#_{l=1}^c\LL_{p_l,q_l}$ tels que 
$$
M=M'
\#^a\RR\PP^3\#^b(\sS^1\times \sS^2).
$$
\item À un nombre fini d'exceptions près, si $X$ est rationnellement connexe et si $M$ admet une fibration de Seifert $M\to B$ dont l'espace d'orbites $B$ est orientable, alors $M$ porte l'une des quatre géométries 
$$
\sS^3,\ \EE^3,\ \sS^2\times \EE^1,\ \Nil.
$$
\end{enumerate}
\item Soit $M = M'
\#^a\RR\PP^3\#^b(\sS^1\times \sS^2)$ une variété compacte sans bord de dimension $3$. Si $M'$ est une variété de Seifert orientable ou une somme connexe d'espaces lenticulaires $M'=\#_{l=1}^c\LL_{p_l,q_l}$, alors il existe une variété projective $X$ définie sur $\RR$, uniréglée et non singulière telle que $M$ soit une composante connexe de $\xr$.
\end{enumerate}
\end{thm}

\begin{proof}
Pour prouver 1.a, il suffit, d'après le théorème de Koll\'ar \ref{thm.kollar}, de réfuter la seule famille infinie d'exceptions~\ref{thm.kollar}.(\ref{kollar.fibrestores}). C'est-à-dire de montrer que si $M\to \sS^1$ est un fibré  localement trivial en tores qui n'admet pas en même temps une fibration de  Seifert, alors $M$ appartient à la liste finie d'exceptions~\ref{thm.kollar}.(\ref{kollar.except}). C'est ce qui est énoncé dans le Corollaire~\ref{cor.mw}. 

Le point 1.b est une conséquence du théorème \ref{thm.cm}. 
Le point 2. est la conjonction des théorèmes~\ref{thm.hm1} et \ref{thm.hm2} ci-dessous.
\end{proof}

\begin{thm}\cite{hm1}\label{thm.hm1}
Toute variété de Seifert orientable est réalisable comme composante connexe de la partie réelle d'une 3-variété projective réelle fibrée en $\mathbb{P}^1$.
\end{thm}
 
 \begin{thm}\cite{hm2}\label{thm.hm2}
Toute somme connexe d'espaces lenticulaires $\#_{l=1}^c\LL_{p_l,q_l}$ est réalisable comme composante connexe de la partie réelle d'une 3-variété projective réelle fibrée en $\mathbb{P}^1$.
\end{thm}

Ces deux théorèmes prouvent la conjecture de Koll\'ar \cite[Conjecture~6.7.(2)]{Ko01}.

\subsubsection*{Programme du modèle minimal sur $\RR$}
L'un des résultats les plus difficiles pour la preuve du théorème \ref{thm.kollar} est un contrôle de la transformation de la topologie à travers le programme du modèle minimal (MMP) sur $\RR$. C'est le théorème ci-dessous qui permet de réduire la question topologique pour les variétés uniréglées aux fibrés de Mori en dimension~$3$ : les fibrés en coniques sur une surface, les fibrations en surfaces rationnelles sur une courbe, et les $3$-variétés de Fano.

\begin{thm}\cite{KoII}\label{thm.mmp}
Soit $X$ une variété algébrique réelle de dimension 3, projective, non singulière et $X^*$ le résultat du MMP sur $\RR$. Supposons le lieu réel $X(\RR)$ orientable.

Alors la normalisation topologique $\overline{X^*(\RR)}\to X^*(\RR)$ (cf. Définition~\ref{dfn.norm}) est une variété linéaire par morceaux et toute composante connexe $L\subset X(\RR)$ s'obtient à partir de $\overline{X^*(\RR)}$ par sommes connexes de composantes de $\overline{X^*(\RR)}$, de copies de $\RR\PP^3$ et de copies de $\sS^2\times \sS^1$.
\end{thm}

\subsubsection*{Théorie symplectique des champs, variétés $\HH^3$ et $\Sol$}
Les théorèmes \ref{thm.vit} concernant $\HH^3$ et \ref{thm.mw} concernant $\Sol$ utilisent la théorie symplectique des champs (SFT) qui par ses méthodes nous emmènerait assez loin de la géométrie algébrique. Par manque de place (et de courage), nous avons renoncé à présenter cet important outil et nous renvoyons le lecteur intéressé aux articles concernés \cite[1.7.5]{EGH} et \cite{mw1}.

\subsubsection*[Suspension d'un difféomorphisme du tore]{Suspension\protect\footnote{Construction appelée aussi "mapping torus".} d'un difféomorphisme du tore}

Notons $z$ une coordonnée sur le cercle 
$\sS^1:=\{\vert z\vert=1\}\subset\CC$, et  $(u,v)$ des coordonnées sur le tore $\sS^1\times\sS^1:=\{\vert u\vert=1,\ \vert v\vert=1\}\subset\CC \times \CC$. Alors $\GLz$ agit sur $\sS^1\times\sS^1$ par 
$$
\left(\begin{matrix}
a&b\\
c&d
\end{matrix}\right) \longmapsto [(u,v) \mapsto (u^av^b,u^cv^d)]
$$

Pour $A\in\GLz$, posons 
$$
M:=\left(S^1\times S^1\right) \times [0,1]/((u,v),0)= (A \cdot(u,v),1).
$$

L'application $\rho\colon M\to \sS^1=[0,1]/(0=1)$ est alors un fibré localement trivial en tores. 
L'espace total $M$ de ce fibré est géométrique, sa géométrie dépend du choix de $A$. Soit $\lambda$ une valeur propre de $A$, on a la trilogie:

\begin{enumerate}
\item  si $\vert \lambda \vert = 1$ et $A$ est périodique, alors $M = \raisebox{-.65ex}{\ensuremath{\Lambda}}\!\backslash \EE^3$,

\item  si $\vert \lambda \vert = 1$ et $A$ est non périodique, alors $M = \raisebox{-.65ex}{\ensuremath{\Lambda}}\!\backslash \Nil$,

\item si $\vert \lambda \vert \ne 1$ (i.e. $A$ hyperbolique), alors  $M = \raisebox{-.65ex}{\ensuremath{\Lambda}}\!\backslash \Sol$. 
\end{enumerate}

Remarquons que dans les deux premiers cas, c'est-à-dire lorsque $M$ porte une géométrie euclidienne ou $\Nil$, $M$ admet aussi une fibration de Seifert, voir Proposition~\ref{prop.seifert}. On en déduit qu'un fibré en tores au-dessus du cercle est soit une variété de Seifert soit une variété $\Sol$.

\subsubsection*{Variétés $\Sol$}\label{p.sol}

Le groupe de Lie $\Sol$ est l'ensemble $\RR^3$ muni du produit semi-direct induit par l'action
$$
\RR\times\RR^2\to \RR^2,\ \left(z,(x,y)\right)  \mapsto \left(e^z x,e^{-z} y \right).
$$

La loi de groupe sur $\RR^3$ est
$$
\left((\alpha,\beta,\lambda),(x,y,z)\right)  \mapsto \left(e^\lambda x + \alpha,e^{-\lambda} y +\beta,z+\lambda\right) 
$$ et la métrique
$$
ds^2=e^{-2z} dx^2 + e^{2z} dy^2+dz^2
$$ 
est invariante à gauche. Le groupe $\Isom(\Sol)$ possède huit composantes connexes et sa composante neutre est $\Sol$ lui-même, cf. \cite[Lemme~3.2]{Tr98}.
Une variété $M$ de dimension $3$ est une \emph{variété $\Sol$} s'il existe un sous-groupe discret d'isométries $\Lambda\subset \Isom(\Sol)$ agissant sans point fixe tel que
$$
M=\raisebox{-.65ex}{\ensuremath{\Lambda}}\!\backslash \Sol
$$
\begin{prop}[Classification des variétés $\Sol$ fermées]

Soit $M$ une variété $\Sol$ compacte sans bord alors l'une des deux assertions suivantes est vérifiée
\begin{enumerate}
\item $M$ est la suspension d'un difféomorphisme hyperbolique.
\item $M$ est un saphir, c'est-à-dire un $\ZZ/2$-quotient d'une variété du type~1.
\end{enumerate}
\end{prop}

\begin{thm}[\cite{mw1}]\label{thm.mw}
Une variété $\Sol$ fermée orientable ne peut se plonger comme composante connexe du lieu réel d'une variété projective non singulière de dimension $3$ admettant une fibration à fibres rationnelles au-dessus d'une courbe.
\end{thm}

Ce théorème prouve essentiellement la conjecture de Koll\'ar \cite[Conjecture~6.7.(1)]{Ko01}.
\begin{cor}\label{cor.mw}
Si une composante connexe orientable $M$ du lieu réel d'une variété projective réelle non singulière de dimension $3$ uniréglée est une variété $\Sol$, alors $X$ est birationnelle à une variété de Fano réelle $Y$ telle que $Y(\RR)$ contient une composante connexe homéomorphe à $M$.
En particulier, il existe au plus un nombre fini de telles variétés $\Sol$ uniréglées (et conjecturalement aucune !).
\end{cor}

\begin{proof}
D'après le Théorème~\ref{thm.mmp}, $X$ est birationnel à une variété réelle $Y$ telle que $Y(\RR)$ contient une composante connexe homéomorphe à $M$ et $Y$ est l'un des trois fibrés de Mori suivant.
\begin{enumerate}
\item Une fibration en coniques au-dessus d'une surface, Koll\'ar a montré dans \cite{KoIII} que $M$ ne peut être $\Sol$;
\item une fibration à fibres rationnelles au-dessus d'une courbe, le Théorème \ref{thm.mw} affirme que $M$ ne peut être $\Sol$;
\item une variété de Fano à singularités terminales et l'on sait \cite{Kawa92} (voir aussi \cite[Section~6]{Ko98}) qu'il n'existe qu'un nombre fini de familles de telles variétés.
\end{enumerate}
\end{proof}

\subsubsection*{Variétés de Fano}

Pour compléter ce panorama, voici trois références concernant les variétés de Fano réelles:
les cubiques réelles de $\PP^4$ ont été classifiées par Krasnov \cite{Kras06,Kras09} et les structures réelles sur la variété de Fano $V_{22}$ ont été classifiées par Koll\'ar et Schreyer \cite{KS04}.

%%%%%%%%%%%%%%%%%%%%%%
\section{Surfaces singulières et paraboles}\label{sec.singsurf}

Dans cette section, basée sur \cite{KoIII,cm1,cm2}, nous donnons une classification des types topologiques possibles pour les surfaces singulières de Du Val (\ref{dfn.duval}) géométriquement rationnelles (voir p.~\pageref{geom.rat}). Pour ce faire, il est commode de considérer une structure d'orbifold à points coniques (\ref{dfn.cone}) sur les composantes connexes de la normalisée topologique (\ref{dfn.norm}). Comme conséquence, nous obtenons une généralisation du théorème de Comessatti (\ref{thm.cm}). Une autre conséquence, qui fut la motivation initiale de cette partie, est la preuve de trois conjectures de Koll\'ar sur les variétés rationnellement connexes (\ref{cor.orbi}). 

\subsubsection*{Surfaces de Du Val}
   
Sur une surface, un point double rationnel est appelé une singularité de Du Val. Ces singularités sont les quotients de $\CC^2$ par les sous-groupes finis de $\SL_2(\CC)$. 

\begin{dfn}\label{dfn.duval}
Une surface projective est une surface \emph{de Du Val} si ses seules singularités sont des  points doubles rationnels.
\end{dfn}

Sur $\CC$, les singularités de Du Val sont classifiées par les diagrammes de Dynkin : il y a 
les \emph{cycliques} $A_\mu$, $\mu \geq 1$, les \emph{diédrales} $D_\mu$,  $\mu \geq 4$, la \emph{tétraédrale} $E_6$, l'\emph{octaédrale} $E_7$ et l'\emph{icosaédrale} $E_8$.
Sur $\RR$, il y a beaucoup plus de possibilités et nous nous limiterons ici à deux séries de singularités cycliques.

Une surface réelle $X$ admet une singularité $A^\pm_\mu$ en un point $P\in X(\RR)$ si au voisinage de $P$, $X$ est $\RR$-analytiquement isomorphe à
$$
x^2 \pm y^2 - z^{\mu + 1} = 0,\ \mu \geq 1
$$

\begin{figure}[ht]
\centering
$A^+_\mu, \mu$ pair \quad \includegraphics[scale=0.8]{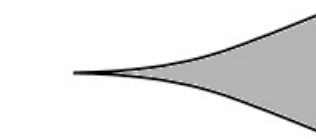}\qquad   $A^+_\mu, \mu$ impair \quad \includegraphics[scale=.5]{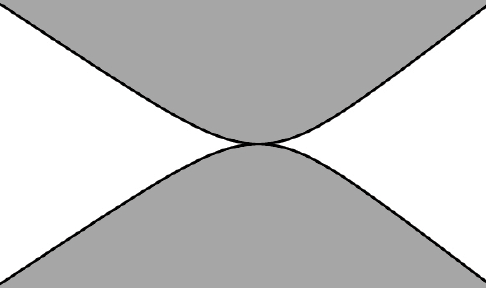}
\caption{$A^+_\mu$, \ $x^2 + y^2 - z^{\mu + 1} = 0,\ \mu \geq 1$}
        \label{fig.a+}
\end{figure}

\begin{figure}[ht]
\centering
$A^-_\mu, \mu$ pair \quad  \includegraphics[scale=0.8]{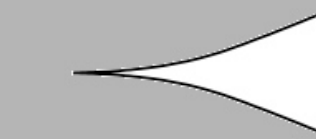} \qquad $A^-_\mu, \mu$ impair \quad \includegraphics[scale=.5]{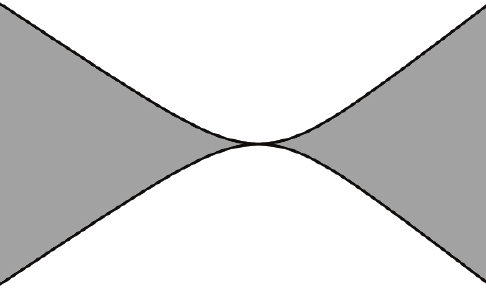}
\caption{$A^-_\mu$, \ $x^2 - y^2 - z^{\mu + 1} = 0,\ \mu \geq 1$}
        \label{fig.a-}
\end{figure}

La partie grisée de la figure~\ref{fig.a+} représente la zone du plan $\RR^2_{z,x}$ où $z^{\mu + 1}-x^2$ est positif. La surface $X$ qui est localement revêtement double du plan ramifié en la courbe $z^{\mu + 1}-x^2=0$ possède des points réels exclusivement au-dessus de cette zone.

Signalons que deux singularités de noms différents ne sont pas isomorphes à l'exception de $A_1^+$ et $A_1^-$.

\begin{figure}[ht]
\centering
\includegraphics[scale=0.4]{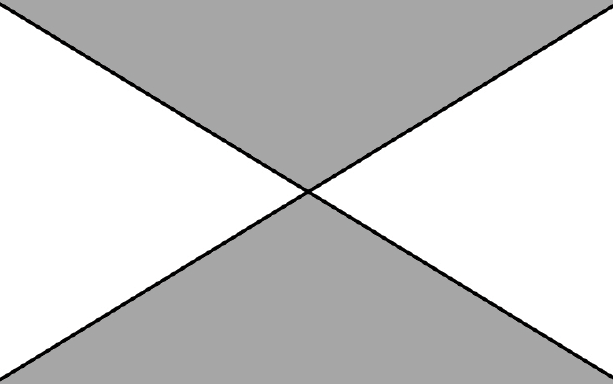} \qquad \includegraphics[scale=0.4]{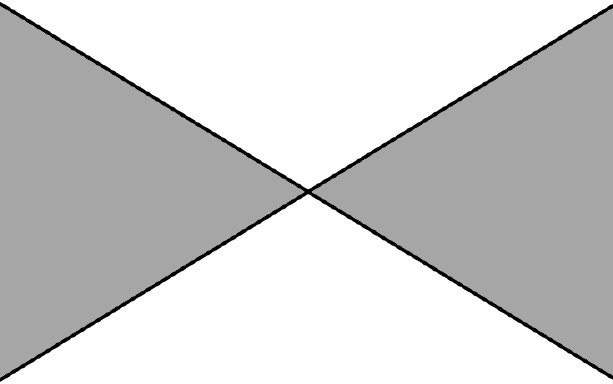}
\caption{$A_1^+\cong A_1^-$}
        \label{fig.a1}
\end{figure}

\subsubsection*{Orbifolds de dimension 2}\label{p.orbi}

On peut justifier ce terme d'orbifold par le fait qu'en anglais, une variété topologique $M$ de dimension $n$ est appelée "$n$-manifold", dans ce cas, $M$ est muni d'une famille de cartes $(\widetilde U,\phi)$ où $\widetilde U$ est un ouvert et $\phi$ un homéomorphisme sur un ouvert $U\subset \RR^n$.
   
Un $n$-orbifold est muni d'un atlas où cette fois une carte $\phi\colon \widetilde U \to U\subset \RR^n$ est un revêtement ramifié fini. On retrouve les variétés si tous les $\phi$ sont de degré~$1$. Plus précisément, chaque ouvert de carte est munie de l'action d'un groupe fini $G$ et $\phi$ se factorise par un homéomorphisme $ \raisebox{-.65ex}{\ensuremath{G}}\!\backslash \widetilde U \to U$ \cite[Chapter~2]{BMP03}.

\begin{dfn}\label{dfn.cone}
 Si $G$ est  cyclique et agit par rotation d'angle $2\pi/\mu$, le point fixe est appelé
un \emph{point conique} d'indice $\mu$.  
\end{dfn}

En général, un orbifold n'est pas homéomorphe à une variété.
Mais en dimension~$2$, tout orbifold $M$ est homéomorphe à une variété topologique notée $\vert M \vert$.

Pour $(p,q)=1$, on note $\sS(p,q)$ l'orbifold de surface lisse sous-jacente $\vert \sS(p,q) \vert=\sS^2$ avec deux points coniques d'indices $p$ et $q$ respectivement. 

\subsubsection*{Normalisation topologique}
Pour traiter de la situation où le lieu réel est singulier, il est commode d'introduire une notion imitant la "séparation des branches" en géométrie algébrique (cf. \cite{KoII}).

\begin{dfn}\label{dfn.norm}
Soit $V$ un complexe simplicial dont le lieu singulier $\Sing(V)$ est fini. La normalisation topologique 
$
\overline{\nu} \colon \overline{V} \to V$ 
est l'unique application propre et continue qui vérifie:

\begin{enumerate}
\item $\overline{\nu}$ est un homéomorphisme au-dessus de $V\setminus \Sing(V)$,

\item si $P\in \Sing(V)$, la fibre $\overline{\nu}^{-1}(P)$ est en bijection avec l'ensemble des composantes connexes locales de  $V$ au voisinage de $P$.
\end{enumerate}
\end{dfn}

Soit $M \subset \overline{\xr}$ une composante connexe de la normalisée topologique du lieu réel d'une surface de Du Val. On munit $M$ d'une structure d'orbifold dont les points coniques d'indice $\mu$ correspondent aux points singuliers de type $A^\pm_\mu$ avec $\mu$ impair vérifiant certaines conditions (voir \ref{dfn.glob}).

\subsection*{Généralisation du Théorème de Comessatti}

Soit $X$ une surface algébrique réelle géométriquement rationnelle et $M \subset \overline{\xr}$  une composante connexe de la normalisée topologique du lieu réel.

\begin{thm}[Comessatti, 1914]
Si $X$ est non singulière et si $M$ est orientable, alors
$M$ est une sphère ou un tore.
\end{thm}

\begin{thm}[\cite{cm2}]\label{thm.cm}

Si $X$ est de Du Val et si $M$ est un orbifold orientable, alors
$M$ est sphérique ou euclidien.
\end{thm}

Ce résultat est un corollaire du Théorème~\ref{thm.orbi} ci-dessous.

 Soit $M$ un orbifold compact de dimension $2$ muni d'un revêtement fini global $\widetilde M \to M$ d'ordre $d$ par une surface lisse. Dans ce cas, la caractéristique d'Euler d'orbifold est définie par
 $$
 \chi(M) := \frac 1d \chi(\widetilde M) \in \QQ
 $$
 
Soit
$M$ un $2$-orbifold avec $k$ points coniques, dont les angles des cônes sont
$2\pi/k_j$, $j = 1,\dots , k$ et 
$\vert M \vert$ la surface lisse sous-jacente à $M$, on a 

$$
\chi(M) = \chi(\vert M \vert) - \sum (1- \frac 1{k_j})
$$

En particulier, $M$ est sphérique ou euclidien
si et seulement si $\vert M \vert$ est sphérique ou euclidien et  $ \sum (1- \frac 1{k_j})\leq 2$.

Lorsque $X(\RR)$ est de dimension 2, la normalisée $\overline{X(\RR})$ est clairement une variété topologique et si $P \in X(\RR)$ est un point singulier de type $A^\pm_\mu$ avec $\mu$ impair, alors $\overline{X(\RR)}$ possède deux composantes connexes locales au voisinage de $P$.

\begin{figure}[ht]
\centering
\includegraphics[scale=.5]{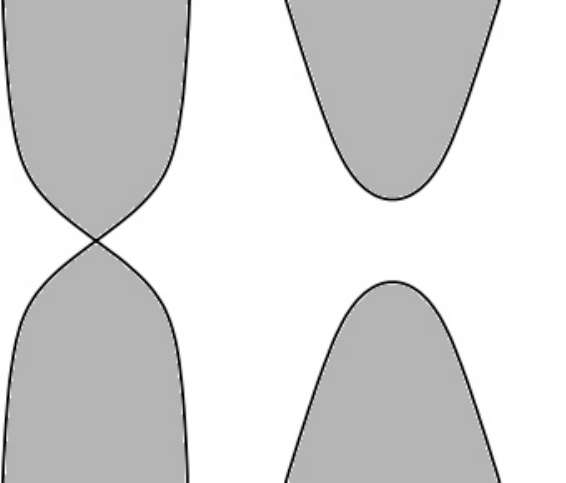}
\caption{$M$ et $\overline{M}$ au voisinage d'un point singulier $A^\pm_\mu$ avec $\mu$ impair.}
        \label{fig.norma}
\end{figure}

\begin{dfn}\label{dfn.glob}
Le point $P$ est \emph{globalement non séparant} si les deux composantes connexes locales au voisinage de $P$  sont sur la même composante connexe de $\overline{X(\RR)}$ et  \emph{globalement séparant} sinon.
\end{dfn}

Soit $X$ une surface réelle de Du Val, $\overline{\nu} \colon \overline{\xr} \to \xr$ la normalisation topologique du lieu réel. On note
 $\Sigma_X$ l'ensemble des points singuliers réels de type $A^-_\mu,\ \mu$ pair ou de type $A^-_\mu,\ \mu$ impair et globalement non séparant. On note $\mathcal{P}_X := \Sing(X)\setminus \Sigma_X$ l'ensemble des autres points singuliers. 

Pour une composante connexe $M\subset \overline {X(\RR)}$, on note $k(M)$ le cardinal $\#\{\overline{\nu}^{-1}(\mathcal{P}_X) \cap M\}$ et $\mu_i(M)$ pour $i=1\dots k(M)$ l'indice d'un point de $\mathcal{P}_X \cap \overline{\nu}(M)$.

 \begin{thm}[\cite{cm1,cm2}]\label{thm.orbi}
Soit $X$ une surface algébrique réelle et $M\subset \overline {X(\RR)}$ une composante connexe de la normalisée topologique du lieu réel. Si $X$ est géométriquement rationnelle, alors

\begin{itemize}
\item $k(M)  \leq 4$,
\item $\sum_{i=1}^k (1- \frac 1{\mu_i+1}) \leq 2$
\item $\vert M \vert =  \sS^1\times \sS^1 \Rightarrow k(M) = 0$

\end{itemize}
\end{thm}

\begin{proof}
Nous discuterons uniquement ici de l'inégalité $k(M)  \leq 4$. La partie principale de la preuve consiste à réduire les cas à certains revêtements doubles du cône quadratique ramifiés le long de courbes singulières de degré 6. Une énumération astucieuse des cas permet alors de conclure. Pour fixer les idées, le lecteur est invité à compter les points doubles de chaque composante connexe sur les figures~\ref{fig:6points} et \ref{fig:5points} qui liste tous les cas possibles avec trois composantes irréductibles qui apparaissent ici comme trois paraboles dans le plan affine. Pour lire ces figures, il faut savoir que deux composantes connexes sont connectées à l'infini si leur bord possède deux arcs non bornés appartenant à la même paire de paraboles.

\begin{figure}[htbp]
\begin{center}
   \includegraphics[scale=.5]{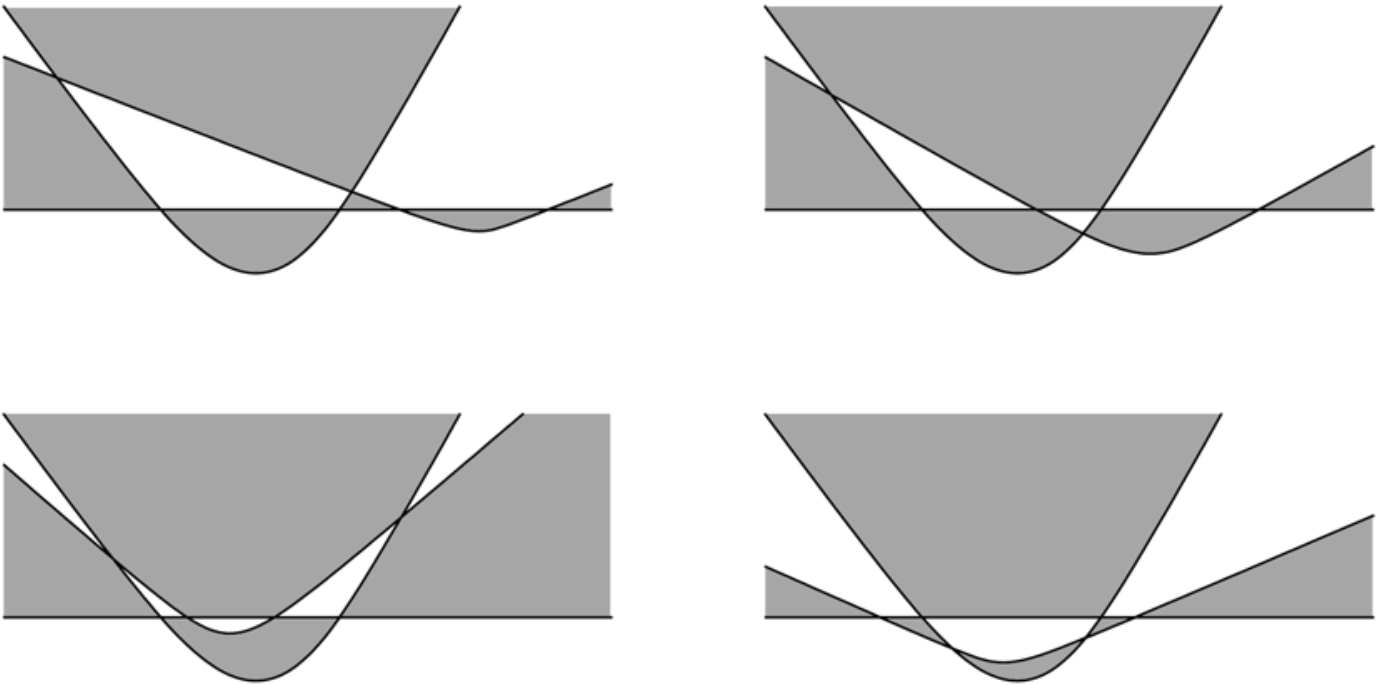}
\caption{$6$ points $A_1$.}\label{fig:6points}
\end{center}
\end{figure}
\begin{figure}[htbp]
\begin{center}
   \includegraphics[scale=.5]{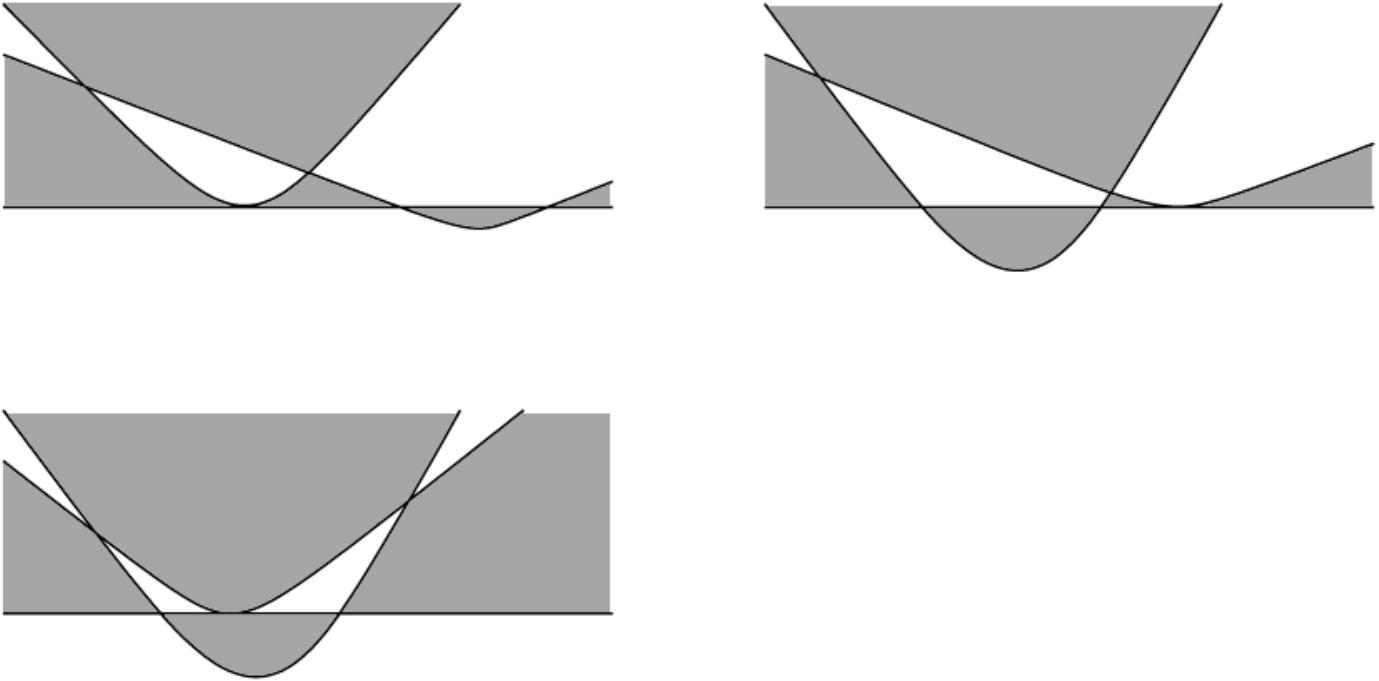}
   \includegraphics[scale=.5]{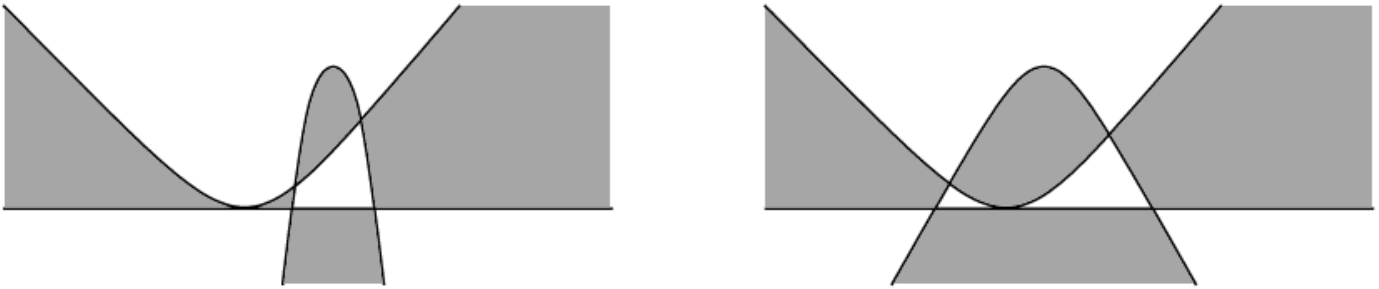}
\caption{$4$ points $A_1$, $1$ point $A_2$.}\label{fig:5points}
\end{center}
\end{figure}
\end{proof}

Si $M\to B$ est une variété de Seifert, on note $k$ le nombre de fibres multiples et pour chaque fibre multiple $i=1\dots k$, $\mu_i$ sa multiplicité. Si $M$ est une somme connexe d'espace lenticulaires, on note $k$ le nombre de lenticulaires et pour chaque lenticulaire $i=1\dots k$, $\mu_i$  l'ordre de son groupe fondamental.

\begin{cor}\label{cor.orbi}
Soit $X\longrightarrow B$ une variété projective réelle de dimension~$3$ fibrée en $\PP^1$ de lieu réel orientable. Si $B$ est une surface géométriquement rationnelle, alors pour toute composante connexe $M\subset X(\RR)$,

\begin{itemize}
\item $k(M)  \leq 4$,
\item $\sum_{i=1}^k (1- \frac 1{\mu_i+1}) \leq 2$
\item $\vert M \vert =  \sS^1\times \sS^1 \Rightarrow k(M) = 0$

\end{itemize}
\end{cor}

Ce résultat prouve la conjecture de Koll\'ar \cite[Conjecture~6.7.(3)]{Ko01} en répondant par l'affirmative aux trois questions \cite[Remark~1.2(1--3)]{KoIII}.
 
 \subsubsection*{Composantes non orientables}

Par éclatement de points réels lisses du plan projectif réel, on réalise toutes les surfaces non orientables comme composantes de surfaces rationnelles réelles non singulières. 
De la même manière, on réalise aisément des orbifolds hyperboliques non orientables 

Lorsque $X$ est géométriquement rationnelle, minimale et non singulière, le Théorème de Comessatti implique que $M$ est sphérique ou euclidien. Dans le cas singulier, nous avons montré qu'il n'en est rien.

\begin{thm}[\cite{cm2}]
Il existe une surface de Du Val géométriquement rationnelle minimale $X$ dont une composante $M  \subset \overline {X(\RR)}$ est un orbifold de type hyperbolique
\end{thm}

\section{Questions et conjectures}\label{sec.conj}
\begin{enumerate}                                                                                                                                                                                                                                                                                      
\item En 2004, j'avais proposé la conjecture suivante à l'occasion de mon habilitation à diriger les recherches.
\begin{conj}
Les composantes réelles {\em géométriques} des variétés uniréglées réelles de
dimension 3 orientables sont exactement les variétés de Seifert
orientables.
\end{conj}

Au vu de \ref{thm.tout}, il y a au plus un nombre fini de contre-exemples éventuels à cette conjecture et pour qu'un tel contre-exemple existe, il faudrait qu'il existe une variété de Fano à singularité terminale dont le lieu réel contient une composante $\Sol$.

\item Nous reprenons une question posée dans \cite{mw1}. Nous avons vu qu'il existe des modèles uniréglés pour toutes les variétés de Seifert, mais la question suivante est ouverte : quel est le modèle projectif réel lisse le plus simples pour une variété $\HH^3$, pour une variété $\Sol$ ? 
\item En complément de $V_{22}$ \cite{KS04}, il serait intéressant d'étudier les structures réelles sur d'autres variétés de Fano.
\item En améliorant la construction de \cite{hm1}, il devrait être possible de montrer que toute variété de Seifert non orientable admet un modèle uniréglé.

\item 
\cite[Section~6]{Ko-simplest} contient une intéressante liste de problèmes ouverts (le problème 61 étant résolu par \ref{thm.vit}), par exemple : 
\begin{enumerate}
\item Quels sont les types topologiques possibles pour une hypersurface de degré $4$ dans $\RR\PP^4$~?
\item Quels sont les types topologiques possibles pour une variété de Calabi-Yau réelle de dimension~$3$~?

\end{enumerate}

\end{enumerate}

%%%%%%%%%%%%%%%%%%%%%%%%%
 \appendix 
 
 \section{Éclatements}\label{anex.eclat}
\subsection*{Éclatement de variétés $\cC^\infty$}
(Cette partie est développée à partir de \cite[2.1]{Mi97}.)

\subsubsection*{Fibré tautologique}
On désigne par $B_n\to \RR\PP^n$ le fibré tautologique au-dessus de l'espace projectif $\RR\PP^n$. Ce fibré a pour fibre en $L\in \RR\PP^n$ la droite vectorielle de $\RR^{n+1}$ représentée par $L$. C'est un fibré réel de rang 1. Rappelons comment en décrire  une trivialisation locale. Soit $v$ un vecteur non nul de $\RR^{n+1}$ et $L\in \RR\PP^n$ la droite vectorielle qu'il engendre. Soit $H\subset \RR^{n+1}$ un hyperplan supplémentaire de $L$. Notons $A\subset  \RR\PP^n$ l'ensemble des droites non contenues dans $H$. Chaque droite $L'$ appartenant au voisinage $A$ de $L$ contient alors exactement un vecteur de la forme $v+w(L')$ avec $w(L')\in H$. On a alors un homéomorphisme 
$$
A\times \RR \to B_n\vert_A,\quad
(L',t)\mapsto (L',tw(L'))
$$
qui est linéaire sur les fibres.

Par construction, $B_n$ est une sous-variété du produit $\RR^{n+1} \times \RR\PP^n$ et le morphisme du fibré tautologique est la restriction de la projection 
$$
\RR^{n+1} \times \RR\PP^n\to \RR\PP^n.
$$

Notons $\pi\colon B_n \to \RR^{n+1}$ la restriction de la projection sur le premier facteur. 
L'application $\pi$ induit alors un difféomorphisme :
$$
B_n\setminus E_P\stackrel{\approx}{\longrightarrow} \RR^{n+1}\setminus \{P\}
$$
où $P=(0,\dots,0)\in \RR^{n+1}$ et $E_P:=\pi^{-1}(P)$.

On dit que $\pi\colon B_n \to \RR^{n+1}$ est \emph{l'éclatement} de $\RR^{n+1}$ centré en $P$. 
La sous-variété $E_P$ de codimension~$1$ dans $B_n$ s'appelle le \emph{diviseur exceptionnel} de l'éclatement. Il découle immédiatement de la définition que $E_P$ est difféomorphe à $\RR\PP^n$.

\subsubsection*{Projectivisé du fibré normal}

Considérons une sous-variété compacte et sans bord $C$ de codimension $r$ dans une variété $M$ de classe $\cC^\infty$. Pour simplifier l'exposition, nous pouvons munir $M$ d'une métrique riemannienne. Soit $\cN\to C$ le fibré normal à $C$ dans $M$ qui est un fibré vectoriel de rang $r$. Notons 
$$
\pi_1\colon E_C\to C
$$ 
le fibré projectivisé de $\cN\to C$. Par définition, la fibre $\pi_1^{-1}(P)$ en $P\in C$ est l'espace projectif des droites de l'espace vectoriel $\cN_P$; $E_C$ est donc l'espace total d'un fibré de fibre $\RR\PP^{r-1}$ au-dessus de $C$.

\subsubsection*{Éclatement d'une variété le long d'une sous-variété}

Rappelons que $C$ étant plongée dans $M$, il existe une application $\cC^\infty$ injective $j\colon\cN\hookrightarrow M$ qui identifie $\cN$ avec un voisinage ouvert $U=j(\cN)$ de $C$ dans $M$. L'injection $j$ est appelée  \emph{voisinage tubulaire} de $C$ dans $M$. L'ouvert $U$ est souvent appelé lui aussi voisinage tubulaire. En retour, $j$ identifie $C$ avec la section nulle de $\cN$. On écrit abusivement $C\subset\cN$ et $j$ induit alors un difféomorphisme $\cN\setminus C \stackrel{\approx}{\longrightarrow} U\setminus C$. 
Notons $\widetilde U$ l'espace total du fibré tautologique au-dessus de $E_C$ et identifions $E_C$ à la section nulle $E_C\subset \widetilde U$. Par construction, $\widetilde U$ est alors une variété de même dimension que $M$ et nous avons un difféomorphisme naturel 
$$
\mu\colon \widetilde U \setminus E_C \stackrel{\approx}{\longrightarrow} U\setminus C
$$
qui s'étend en une application $\cC^\infty$
$$
f\colon \widetilde U \to  U\subset M
$$
telle que 
$f\vert_{E_C}=\pi_1$.

En passant sous silence les choix effectués au cours de cette construction, on donne la définition suivante.
\begin{dfn}
La variété \emph{éclatée} $\widetilde M$ de $M$ le long de $C$ est obtenue en recollant $\widetilde U$ à $M\setminus C$ par le difféomorphisme $\mu$. L'application $\cC^\infty$ 
$$
\pi\colon \widetilde M \to M
$$ 
définie par
$\pi\vert_{M\setminus C}=\id$ et $\pi\vert_{\widetilde U}=f$
est \emph{l'éclatement topologique} de $M$ le long de~$C$.
\end{dfn}

La sous-variété $C\subset M$ est le \emph{centre} de l'éclatement et la sous-variété $E_C$ de codimension~$1$  dans $\widetilde M$ est le \emph{diviseur exceptionnel}. 
On note souvent $B_CM:=\widetilde M$ la variété éclatée.

Si $L\subset M$ est un  sous-ensemble fermé, on dit qu'un sous-ensemble $\widetilde L\subset \widetilde M$ est le \emph{transformé strict} de $L$ si 
\begin{itemize}
\item $\pi(\widetilde L)=L$,
\item $\widetilde L$ est fermé dans $\widetilde M$,
\item $\widetilde L\setminus E_C$ est dense dans $\widetilde L$.
\end{itemize}
Le lecteur intéressé par une discussion plus détaillée est invité à consulter \cite[Section~2]{AK85}.

%%%%%%%%%%
\subsection*{Éclatement de variétés algébriques}

\subsubsection*{Transformée birationnelle}
Considérons maintenant une sous-variété algébrique $W\subset \PP^N$ donnée par $r$ équations $\{f_1=0,\dots,f_r=0\}$. 

\begin{dfn}
La variété \emph{éclatée} de $\PP^N$ le long de $W$ est la sous-variété $B_W\PP^N$ de $\PP^N_{x_0:\dots:x_N}\times\PP^{r-1}_{y_1:\dots:y_r}$ donnée par les $r-1$ équations 
$$
\left\{
\begin{array}{ccc}
y_1f_2(x_0,\dots,x_N)-y_2f_1(x_0,\dots,x_N)=0,\\
y_2f_3(x_0,\dots,x_N)-y_3f_2(x_0,\dots,x_N)=0,\\
\vdots\\
y_{r-1}f_r(x_0,\dots,x_N)-y_rf_{r-1}(x_0,\dots,x_N)=0.
\end{array}
\right.
$$

Et l'éclatement $\pi_W\colon B_W\PP^N \to \PP^N$ est donné par :
 $$
 \left(
 (x_0:\dots:x_N),\ (y_1:\dots:y_r)\right)\mapsto (x_0:\dots:x_N).
 $$
\end{dfn}

Si $\codim W =r$, on retrouve l'interprétation en terme de fibré normal en chaque point lisse de $W$.

Pour  une sous-variété $V\subset \PP^N$, nous notons $\widetilde  V$ l'adhérence de Zariski de $\pi_W^{-1}(V\setminus W\cap V)$ dans $B_W\PP^N$. 

\begin{dfn}\label{app.transf.bir}
La sous-variété $\widetilde  V$ est appelée \emph{transformée birationnelle} (ou \emph{transformée stricte}) de $V$ par $\pi_W$.
\end{dfn}

On peut montrer que cette variété ne dépend pas des plongements de $V\subset \PP^N$ et $W\subset \PP^N$ mais seulement du plongement $W\cap V\subset V$.
\begin{dfn}
Si $W\subset V$, la restriction $\widetilde  V\to V$ de $\pi_W$ est \emph{l'éclatement algébrique} de $V$ le long de $W$ et on note $B_WV  :=\widetilde  V$.
\end{dfn}

Dans le cas où $V$ et $W$ sont lisses, on a un difféomorphisme entre l'éclaté topologique et l'éclaté algébrique qui commute avec les morphismes au-dessus de $V$.

%%%%%%%%%%
\subsection*{Topologie des éclatements}
Revoir si besoin la définition de la somme connexe de deux variétés orientées p.~\pageref{somme.connexe}. Si l'une des deux variétés est non orientable, on définit la somme connexe comme variété non orientable.

\begin{prop}\label{prop.eclat}
Soit $P\in\RR^n$, alors 
$B_P\RR^n$ est difféomorphe à $\RR^n\#\RR\PP^{n-1}$. Plus généralement, si $C\subset M$ est une sous-variété non singulière de codimension $r$ et de fibré normal trivial, $B_CM$ est difféomorphe à $M\#C\times \RR\PP^{r-1}$.
\end{prop}

\begin{ex} À titre d'illustration, nous nous proposons d'expliciter le calcul de l'exemple \ref{ex.blow}.

Notre but étant de déterminer la topologie de la variété éclatée, il est raisonnable de se placer dans la situation $\cC^\infty$ qui permet en particulier de travailler dans un ouvert $U$ aussi "petit" que l'on veut, par exemple un voisinage ouvert de l'unique point réel de la courbe singulière $D$. Nous utiliserons librement l'identification $U\approx \RR^n$. En analytique, une telle identification est impossible ($\CC^n$ n'étant pas biholomorphe à l'un de ses ouverts stricts) et en algébrique c'est encore pire à cause de la faiblesse de  la topologie de Zariski. Voir par exemple~\cite[Chapter~VI \S2.2]{Sc94} pour une façon plus sérieuse d'expliquer ce qu'il se passe ici.

Soit $D$ la courbe d'équations $(z=x^2+y^2=0)$ dans $\CC^3$. Nous allons calculer successivement $\pi_1\colon Y_1=B_D\RR^3\to \RR^3$ et $\pi\colon Y=B_PY_1\to \RR^3$ où $P\in Y_1$ est l'unique point singulier de $Y_1$.

Par définition, $B_D\RR^3\subset\RR^3_{x,y,z}\times\PP^{1}_{\alpha:\beta}$ est déterminé par l'équation
$$
\alpha(x^2+y^2)-\beta z=0
$$
qui possède un unique point singulier $P$ appartenant à la carte affine $\alpha\ne0$.
En restriction à cette carte, $Y_1$ est l'hypersurface affine d'équation $x_1^2+y_1^2-z_1t_1=0$ où $x_1=x,y_1=y,z_1=z,t_1=\beta$  et $P=(0,0,0,0)$.

Pour calculer $B_PY_1$, on éclate $\pi_P\colon \widetilde{\RR^4}\to\RR^4$ au point $P$ et on considère le transformé strict $Y=B_PY_1$ de $Y_1$.

Les quatre équations de  $\pi_P^{-1}(Y_1)\subset \RR^4_{x_1,y_1,z_1,t_1}\times\PP^3_{a:b:c:d}$ étant 
$$
x_1^2+y_1^2-z_1t_1=ay_1-bx_1=bz_1-cy_1=ct_1-dz_1=0.
$$

En restriction à la carte $c\ne0$, $\pi_P^{-1}(Y_1)$ est la variété affine d'équation $z_2^2(x_2^2+y_2^2-t_2)=0$ dans  le sous-espace affine de $\RR^7_{x_2,y_2,z_2,t_2,a,b,d}$ d'équations
$$
\left\{
\begin{array}{c}
x_2=a,\\
y_2=b,\\
t_2=d.
\end{array}
\right.
$$
où $z_2=z_1, x_2=x_1/z_1,y_2=y_1/z_1,t_2=t_1/z_1$. 

La trace du diviseur exceptionnel sur cette carte est déterminée par $z_2=0$, et on en déduit que $Y$ a pour équation $x_2^2+y_2^2-t_2=0$ dans le sous-espace affine
$x_2-a=y_2-b=t_2-d=0$. Le lieu réel est donc la variété produit d'un paraboloïde de révolution avec la droite $\RR$. Dans la carte $d\ne 0$ la situation topologique est identique et il nous reste à nous convaincre que le recollement donne le produit 
$$
\sS^2\times \sS^1.
$$
\end{ex}

 \section{Géométrisation et classification}\label{anex.dim3}
 
Plusieurs articles sont parus dans la Gazette à propos de la Conjecture de Poincaré et de la conjecture de Géométrisation de Thurston \cite{An05}, \cite{Mi04}, \cite{Be05}, \cite{Be13}. En nous basant sur 
l'introduction de \cite{Be-3mfd-10}, nous rappelons très rapidement l'état de l'art depuis que les travaux de Perelman ont ouvert la voie à une classification complète.

\subsubsection*{Classification des variétés $\cC^\infty$ de dimension $3$}

Dans ce paragraphe, une \emph{variété} de dimension $n$ peut posséder un \emph{bord}, noté $\partial M$ caractérisé par le fait que tout point $p\in\partial M$ possède un voisinage dans $M$ qui est localement homéomorphe au produit $\RR^{n-1}\times \RR_{\geq 0}=\{(x_1,\dots,x_n)\in\RR^n\st  x_n\geq 0\}$ avec $p$ identifié à $(0,\dots,0)$. S'il est non vide, le bord d'une variété de dimension $n$ est une variété de dimension $n-1$ dont le bord est vide: $\partial \partial M = \emptyset$.
\emph{L'intérieur} d'une variété à bord est la sous-variété complémentaire du bord $M\setminus \partial M$. Tout point $p\in M\setminus \partial M$ possède un voisinage homéomorphe à $\RR^n$. Une variété \emph{fermée} est une variété compacte telle que $\partial M =\emptyset$.

Une surface fermée connexe $S$ dans une variété $M$ de dimension $3$ orientable et compacte est \emph{essentielle} si son groupe fondamental s'injecte dans $\pi_1(M)$ et si $S$ ne borde pas une $3$-boule ni ne coborde un produit avec une composante connexe de $\partial M$.
%\marginpar{Ajouter Hamilton, Flot de Ricci ?}

On peut énoncer la célèbre conjecture de géométrisation de Thurston (Voir \ref{dfn.geom} pour la définition de \emph{variété géométrique}).
\begin{conj}[Conjecture de géométrisation de Thurston]\label{conj.geom}
L'intérieur d'une variété de dimension $3$  orientable et compacte peut être découpée le long d'une collection finie de $2$-sphères et de $2$-tores plongés essentiels, disjoints deux-à-deux, en une collection canonique de $3$-variétés géométriques après avoir bouché les bords sphériques par des $3$-boules.
\end{conj}

Chaque composante connexe du complémentaire de la famille de tores et sphères admet une métrique localement homogène de volume fini. Soit $M$ une telle composante connexe, et $\widehat{M}$ la variété compacte obtenue en "bouchant" les trous. Alors $M$ possède une structure géométrique modelée sur l'une des huit géométries de la page \pageref{p.8geom}.

Cette conjecture contient comme cas particulier la non moins célèbre conjecture de Poincaré.
\begin{conj}[Conjecture de Poincaré]
Soit $M$ une variété de dimension $3$ simplement connexe et fermée alors $M$ est difféomorphe à la 
$3$-sphère.
\end{conj}

En réunissant le théorème d'hyperbolisation de Thurston \cite[1.1.5]{Be-3mfd-10} et le théorème de Perelman \cite[1.1.6]{Be-3mfd-10} on obtient la classification pour les variétés orientables.

\begin{thm}[Théorème de géométrisation]
La conjecture de géométrisation \ref{conj.geom} est vraie pour toute variété compacte orientable de dimension~$3$.
\end{thm}

\backmatter

%%%%%%%%%%%%%%%%%%%%%%%
\bibliographystyle{smfalpha}
\bibliography{biblio-perso,biblio-tvar,biblio-preprints}
%%%%%%%%%%%%%%%%%%%%%%%
\end{document}